\documentclass[mathpazo]{cicp}
\usepackage{graphicx}
\usepackage{bm}
\usepackage{color}
\usepackage{booktabs}
\usepackage{multirow}
\DeclareMathOperator{\diag}{diag}

\definecolor{lzcol}{rgb}{1, 0, 0}

\begin{document}
\title{Fast Algorithm For Solving Time-dependent Multiscale radiative transport Equation}


\author{Qinchen Song\affil{1}, Lei Zhang\affil{1}\corrauth, Min Tang\affil{1}\corrauth}
\address{\affilnum{1}\ Department of Mathematics, Institute of Natural Sciences,
Shanghai Jiao Tong University, Minhang, Shanghai 200240, P.R. China}
\email{{\tt sqc9931@sjtu.edu.cn}, {\tt lzhang2012@sjtu.edu.cn}, {\tt tangmin@sjtu.edu.cn}}


\begin{abstract}
When solving the time-dependent radiative transport equation (RTE), implicit time discretization is often employed for its robustness and stability. This results in a sequence of steady-state RTEs with identical cross-sections but varying source terms, whose repeated solution is computationally costly. To address this, we first apply the adaptive tailored finite point scheme (TFPS) for spatial discretization. This scheme exploits prior knowledge of the background media's optical properties to adaptively compress the angular domain, constructing a compressed linear system. A key feature is its ability to reconstruct the layer structure after compression, faithfully capturing the variance at the layer. We then use the Recursive Skeleton Method (RSM) to obtain an explicit multilevel decomposition of the inverse discrete operator, which is reused for all steady-state solutions. Numerical experiments show that our framework achieves high accuracy and significant efficiency across diverse scenarios.
\end{abstract}


\maketitle

\section{Introduction}
\label{sec:intro}
The radiative transport equation (RTE) describes the propagation and interaction of particles, such as photons and neutrons, with background media. It finds numerous applications in nuclear engineering, atmospheric science, astrophysics, remote sensing, optical tomography, thermal engineering, etc. The time-dependent RTE reads as follows:
\begin{equation}\label{eq:sec1:1}
    \frac{\partial{\psi}(\mathbf{z},\mathbf{u},t)}{\partial t}+\frac{\mathbf{u}}{\epsilon}\cdot\nabla \psi(\mathbf{z},\mathbf{u},t)+\frac{\sigma_{T}}{\epsilon^2}(\mathbf{z})\psi(\mathbf{z},\mathbf{u},t)=(\frac{\sigma_{T}}{\epsilon^{2}}-\sigma_{a})(\mathbf{z})\int_{S}\kappa(\mathbf{u},\mathbf{u}')\psi(\mathbf{z},\mathbf{u},t)d\mathbf{u}'+q(\mathbf{z},t), 
\end{equation}
where $\mathbf{z}\in \Omega \subset\mathbb{R}^{3}$, $\mathbf{u}\in\mathbb{S}^{2}=\{\mathbf{u}\in\mathbb{R}^{3}|\Vert\mathbf{u}\Vert=1\}$ and $t\in [0,T]$ represent respectively the location, the moving direction of the particles and the current time, with $\Vert\cdot\Vert$ being the Euclidean norm. The angular flux $\psi(\mathbf{z},\mathbf{u},t)$ represents the density of particles moving along the direction $\mathbf{u}$ at location $\mathbf{z}$, time $t$. The dimensionless parameter $\epsilon \in (0,1]$ denotes the ratio of the particle mean free path to the characteristic length of the computational domain, serving as a scaling factor. Following the diffusive scaling framework in \cite{han2014two}, the total cross section, scattering cross section, absorption cross section, and external source are rescaled as $\frac{\sigma_T}{\epsilon}$, $\frac{\sigma_T}{\epsilon} - \epsilon\sigma_a$, $\epsilon\sigma_a$, and $\epsilon q$, respectively. Consistent with \cite{han2014two}, we assume that $\sigma_T$, $\sigma_a$, and $q$ are non-negative, spatially dependent, and piecewise smooth and bounded. Moreover, their spatial gradients $\nabla\sigma_T$, $\nabla\sigma_a$, and $\nabla q$ are uniformly bounded in the $L^\infty$ norm ,except at the interfaces. This regularity condition ensures that the spatial discretization can adequately resolve the spatial variations of the material parameters. Here 
The kernel function $\kappa(\mathbf{u},\mathbf{u}')$ provides the transitional probability for particles moving in direction $\mathbf{u}'$ to be scattered into direction $\mathbf{u}$. One typical example is the Henyey-Greenstein (HG) function \cite{henyey1941diffuse}, which depends solely on the inner product $\mathbf{u}\cdot\mathbf{u}'$ of directions $\mathbf{u}$ and $\mathbf{u}'$ and is defined as follows:
\begin{equation}
\kappa(\mathbf{u},\mathbf{u}')=\frac{1-g^{2}}{(1+g^{2}-2g\mathbf{u}\cdot\mathbf{u}')^{3/2}}.
\end{equation}
Here the parameter $g\in[-1,1]$ is the anisotropy factor and is used to characterize the angular distribution of scattering. The initial condition for equation \eqref{eq:sec1:1} is prescribed as:
\begin{equation}\label{eq:sec1:2}
    \psi(\mathbf{z},\mathbf{u},0) = \Psi^{0}(\mathbf{z},\mathbf{u}).
\end{equation}
For the boundary conditions, we impose Dirichlet conditions specifying the inflow angular flux:
\begin{equation}\label{eq:sec1:3}
\begin{aligned}
    \psi(\mathbf{z},\mathbf{u},t)=\Psi(\mathbf{z},\mathbf{u},t),\quad (\mathbf{z},\mathbf{u})\in\Gamma^{-}=\{\mathbf{z}\in\Gamma=\partial\Omega,  \mathbf{u}\in\mathbb{S}^{2}| \mathbf{u}\cdot\mathbf{n}_{\mathbf{z}}<0 \},  
    \end{aligned}
\end{equation}
where $\mathbf{n}_{\mathbf{z}}$ denotes the outward unit normal vector to $\Omega$ at the point $\mathbf{z}$, and $\Gamma^{-}$ is the inflow boundary. Besides, the solution $\psi$ must remain continuous across material interfaces, i.e. for any interface $\alpha$, 
\begin{equation}\label{eq:sec1:4}
    [\psi(\mathbf{z},\mathbf{u},t)]|_{\alpha}=0,
\end{equation}
with $[\cdot]|_{\alpha}$ representing the jump in the function value across interface $\alpha$.

The numerical solution of the RTE is notoriously challenging due to its high-dimensional phase space, which yields prohibitively large discretized systems. This difficulty is compounded by the inherent multiscale nature of the problem, governed by the scaling parameter $\epsilon$. In heterogeneous media, $\epsilon$ spans orders of magnitude from $\mathcal{O}(10^{-5})$ (diffusive regime) to $\mathcal{O}(1)$ (transport regime). This multiscale behavior imposes stringent requirements on spatial discretization: to ensure solution accuracy and stability, extremely fine meshes are often required \cite{lewis1984computational,reed1971new,larsen1987asymptotic,larsen1989asymptotic,adams2001discontinuous}, further exacerbating the size of the resulting linear systems. For time-dependent problems, as $\epsilon \to 0$, the particle interaction terms  dominate the dynamics, introducing a fast time scale $O(\epsilon)$ that coexists with the slow macroscopic evolution $O(1)$. This wide separation of time scales renders the system numerically stiff, restricting explicit schemes to impractically small steps $\Delta t \lesssim \mathcal{O}(\epsilon)$ for stability \cite{jin1999efficient}. To circumvent this stability constraint, fully implicit temporal discretizations are standard. Consequently, each time step requires the iterative solution of a large linear system. The cumulative effect renders the direct numerical simulation of time-dependent, multiscale RTE problems computationally prohibitive, thereby motivating the development of efficient, scale-robust numerical algorithms.

In \cite{AdaptiveTFPS}, we developed the adaptive tailored finite point scheme (ATFPS) for the steady-state RTE. By exploiting the low-rank structure of the angular dependence, ATFPS efficiently compresses the angular degrees of freedom while preserving accuracy through explicit reconstruction of boundary and interface layers, thereby overcoming a key limitation of conventional methods. This approach significantly reduces the computational cost for steady-state problems and naturally extends to the time-dependent setting. Specifically, we define the steady-state transport operator as
\[
 \mathcal{L}_{\epsilon}\psi = \frac{\mathbf{u}}{\epsilon}\cdot\nabla \psi + \frac{\sigma_{T}}{\epsilon^2}\psi - \left(\frac{\sigma_{T}}{\epsilon^{2}} - \sigma_{a}\right)\int_{S}\kappa(\mathbf{u},\mathbf{u}')\psi \, d\mathbf{u}'.
\]
The associated solution operator $\mathcal{L}_{\epsilon}^{-1}$ maps a source term $f$ and an incoming boundary condition $\Psi$ to the unique solution $\psi$ of
\begin{equation}\label{eq:sec1:5}     
\mathcal{L}_{\epsilon}\psi = f, \quad \psi|_{\Gamma^{-}} = \Psi.
\end{equation}
Combined with the discrete ordinate method (DOM), ATFPS yields a low-rank approximation $\mathcal{L}_{\epsilon, h, \delta}^{-1}$ of $\mathcal{L}_{\epsilon}^{-1}$, where $h$ is the mesh size and $\delta$ is the basis truncation threshold. This approximation achieves first-order accuracy in $\delta$ and second-order accuracy in $h$, uniformly across boundary layers. Moreover, the compression ratio is governed by $\delta$ and the local rank structure determined by the optical properties $(\sigma_T, \sigma_a, \epsilon)$. For instance, in the diffusion regime ($\epsilon \ll 1$), the operator is highly compressible, whereas compressibility diminishes in the transport regime ($\epsilon = \mathcal{O}(1)$).

By employing a fully implicit temporal discretization, the angular compression framework developed for the steady-state RTE extends naturally to the time-dependent setting. After time discretization, the compressed solution operator $\mathcal{L}_{\epsilon, h, \delta}^{-1}$ must be applied repeatedly during the iterative linear solutions at each time step. To avoid the prohibitive cost of repeatedly solving the boundary-value problem \eqref{eq:sec1:5}, an explicit representation of $\mathcal{L}_{\epsilon, h, \delta}^{-1}$ is required—one that can be applied directly to the right-hand side vector at each iteration.
To construct this operator efficiently, we employ a multilevel framework based on the recursive skeleton method (RSM) \cite{ho2012fast,minden2017recursive}. While RSM has proven effective for fast solvers of steady-state RTE integral formulations \cite{fan2019fast, ren2019fast}, this work extends the methodology to adaptively compressed time-dependent systems. The construction hinges on the locality (or approximate locality) of the basis functions, a property inherently satisfied by our discretization. Exploiting this structure, we iteratively decompose the fine-grid solution space into a hierarchy of nested subspaces. This decomposition induces a corresponding multilevel factorization of the solution operator, expressed via projection and restriction operators between these nested function spaces. Notably, this construction provides a novel reinterpretation of RSM within the framework of nested solution spaces. Using this framework, the explicit representation of the solution operator can be constructed with $\mathcal{O}(I^3)$ operations, and applied to a vector at a cost of $\mathcal{O}(I^2 \log I)$ for an $I \times I$ spatial grid.

The inherent low-rank structure of the angular dependence in the RTE has motivated numerous model reduction strategies. Representative approaches include the reduced basis method (RBM) \cite{peng2022reduced}, proper orthogonal decomposition (POD) \cite{buchan2015pod, hughes2020discontinuous, hughes2022adaptive}, dynamical low-rank approximation (DLRA) \cite{einkemmer2021asymptotic, peng2021high, peng2020low, peng2023sweep}, and randomized SVD \cite{chen2020random, chen2021low}. While effective in smooth regimes, these techniques typically rely on \textit{global} basis functions constructed from solution snapshots, which struggle to resolve localized boundary and interface layers without excessive global ranks. Consequently, accuracy degrades sharply in multiscale or heterogeneous settings. 
Besides, fast hierarchical techniques such as recursive skeleton method \cite{ho2012fast, minden2017recursive} have accelerated the computation for steady-state integral formulations \cite{fan2019fast, ren2019fast}, but extending them to time-dependent settings with angular compression remains an open challenge. 
In this work, we bridge these gaps by applying ATFPS and integrating it into RSM. Our approach has twofold novelty:

\begin{itemize}
    \item \textbf{Adaptive angular compression with layer preservation.} We integrate ATFPS for spatial discretization, which exploits local low-rank properties to enable adaptive compression in the angular domain while fully preserving boundary and interface layer structures.
    \item \textbf{Multilevel decomposition for repeated use.} We employ RSM to construct a multilevel decomposition of the compressed solution operator, which can be applied consistently across different time steps. Moreover, this decomposition is invariant to changes in (1) time step sizes, (2) source terms, and (3) initial/boundary conditions.
\end{itemize}
\vspace{-1cm}
\paragraph{Outline}
The paper is organized as follows: Section \ref{sec:discretization_scheme} introduces the discretization framework for the time-dependent RTE, which combines the midpoint rule for temporal discretization, the discrete ordinate method (DOM) for angular discretization, and the ATFPS for spatial discretization. This process yields a sequence of compressed linear systems that share a common discrete operator. Section \ref{sec:RS} presents the recursive skeleton method, which constructs a multilevel decomposition of the compressed inverse operator. Section \ref{sec:experiments} validates the accuracy and computational efficiency of the proposed scheme through a series of numerical experiments. Finally, Section \ref{sec:conclusion} summarizes our main contributions and discusses potential directions for future research.

\section{Discretization scheme}
\label{sec:discretization_scheme}
For clarity, we restrict the original formulation \eqref{eq:sec1:1} to the time-dependent RTE in x–y geometry as follows:
\begin{equation}
\partial_{t}\psi(x, y, c, s, t)+\mathcal{L}_{\epsilon}\psi = q(x, y, t),\quad (x,y)\in\Omega\subset \mathbb{R}^{2},\ (c,s)\in\mathbb{D}=\{(x,y)\in\mathbb{R}^{2}|x^{2}+y^{2}\leq 1\}.
\end{equation}
where $c$ and $s$ are the projections of the angular direction vector  $\mathbf{u}$ in \eqref{eq:sec1:1} onto the x and y axes. The operator $\mathcal{L}_{\epsilon}$ is the natural restriction of the one in \eqref{eq:sec1:1} to the x-y plane, with the form:
\begin{equation}\label{eq:sec2:40}
\mathcal{L}_{\epsilon}\psi=\frac{c}{\epsilon}\partial_x\psi+\frac{s}{\epsilon}\partial_y\psi+\frac{\sigma_{T}(x, y)}{\epsilon^2}\psi-(\frac{\sigma_{T}(x,y)}{\epsilon}-\sigma_{a}(x, y))\int_{\mathbb{S}^{2}}\kappa(c, s, c', s')\psi(x, y, c, s, t)du_xdu_y.
\end{equation}

In this section, we present the discretization scheme for the time-dependent RTE in x-y geometry. Specifically, we adopt the midpoint method for temporal discretization, the discrete ordinate method (DOM) for angular discretization, and the ATFPS for spatial discretization. The proposed methodology extends naturally to physically relevant three-dimensional settings.

\subsection{Time discretization}
For time-dependent problems, as $\epsilon \to 0$, the CFL condition forces explicit time-stepping schemes to obey $\Delta t \lesssim \mathcal{O}(\epsilon)$ \cite{jin1999efficient}. To overcome this stability limitation, fully implicit temporal discretizations are commonly adopted. For better accuracy, we employ the second-order midpoint method. Let $\Delta t$ denote the time step size and $N$ the total number of steps, so that the discrete time points are $t_n = n\Delta t$ for $0 \le n \le N$. Furthermore, we denote the discrete source term as $q^{n-1/2} = q(x,y,t_{n-1/2})$ and the numerical approximation to the angular flux as $\psi^n(x,y,c,s) \approx \psi(x,y,c,s,t_n)$.
Using the midpoint scheme, we derive the following semi-discrete system:
\begin{equation}\label{eq:sec2:103}
    \frac{\psi^{n}-\psi^{n-1}}{\Delta t}+\mathcal{L}_{\epsilon}\frac{\psi^{n}+\psi^{n-1}}{2}=q^{n-1/2}, \quad 1\leq n\leq N, \  n \in\mathbb{Z}.
\end{equation}

According to equations \eqref{eq:sec1:2}, \eqref{eq:sec1:3} and \eqref{eq:sec1:4}, the initial condition, boundary condition, and interface condition after time discretization are as follows. 

The initial condition is given by:
\begin{equation}\label{eq:sec2:106}
    \psi^{0}(x,y,c,s)=\Psi^{0}(x,y,c,s),\quad (x, y)\in\Omega.
\end{equation}
The Dirichlet boundary condition takes the form:
\begin{equation}\label{eq:sec2:107}
    \psi^{n}(x,y,c,s)=\Psi(x,y,c,s,t_{n})\triangleq\Psi^{n},\quad (x,y,c,s)\in\Gamma^{-},\quad 1\leq n\leq N.
\end{equation}
The interface condition for any interface $\alpha$ is expressed as:
\begin{equation}\label{eq:sec2:108}
    [\psi^{n}(x,y,c,s)]|_{\alpha}=0,\quad 1\leq n\leq N.
\end{equation}

Given the solution $\psi^{n-1}$ at time step $n-1$, we compute $\psi^{n}$ by solving the coupled system consisting of equations \eqref{eq:sec2:103}, \eqref{eq:sec2:107}, and \eqref{eq:sec2:108}. Beginning with the initial condition $\psi^{0}$ specified in \eqref{eq:sec2:106}, this time-marching scheme generates the complete temporal solution sequence $\{\psi^{n}\}_{n=1}^{N}$ through successive time steps. 

\subsection{Angular discretization by DOM}
The DOM approximates the continuous angular domain by a finite set of discrete velocity directions and replaces the scattering integral in the RTE with an appropriate numerical quadrature. The construction and accuracy of such quadrature rules have been extensively studied; see, for example, \cite{azmy2010advances, lewis1984computational}.
We denote the quadrature set for the RTE in the x-y geometry by
\[
\{c_{m},s_{m},\omega_{m}\}_{m\in\mathcal{M}},\quad \mathcal{M}=\{1,2,\dots,4M\}.
\]
where $\mathbf{u}_{m} = (c_m, s_m)$ represents the $m$-th discrete propagation direction and $\omega_m$ is its associated quadrature weight. An example of DOM in x-y geometry is illustrated in Figure~\ref{fig:DOM}.

\begin{figure}[htbp]
    \centering    
    \includegraphics[width=0.6\textwidth]{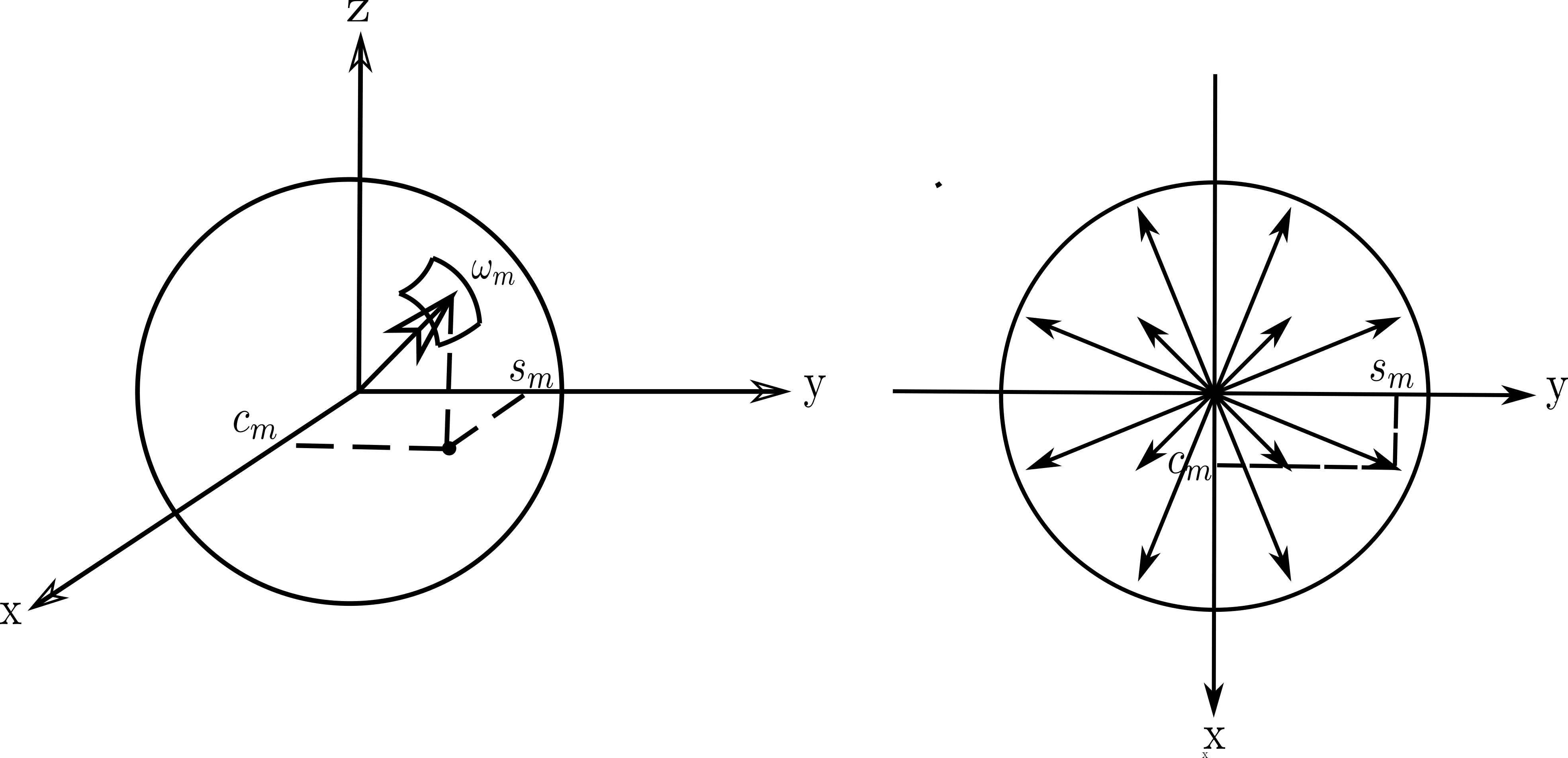}
    \caption{DOM in x-y geometry. Left figure: a quadrature point in three spatial dimensional case. Right figure: example of DOM ($S_{N}$) in x-y geometry with $N=2$.}
    \label{fig:DOM}
\end{figure}

After discretizing the angular domain using DOM, equation~\eqref{eq:sec2:103} takes the form:
\begin{equation}\label{eq:sec3:23}
    \frac{\bm{\psi}^{n}-\bm{\psi}^{n-1}}{\Delta t}+\mathcal{L}_{\epsilon, \mathcal{M}}\frac{\bm{\psi}^{n}+\bm{\psi}^{n-1}}{2}=q^{n-1/2},\quad \ 1\leq n\leq N,\ n\in\mathbb{Z}.
\end{equation}
Here, $\bm{\psi}^{n}=(\psi_{m}^{n})_{m\in\mathcal{M}}$ is a vector-valued function with components satisfying $\psi_{m}^{n}\approx\psi^{n}(x, y, c_m, s_m)$. Based on the definition of  $\mathcal{L}_{\epsilon}$ in \eqref{eq:sec2:40}, the operator $\mathcal{L}_{\epsilon, \mathcal{M}}$ is defined as follows: for any vector-valued function $f=(f_{m})_{m\in \mathcal{M}}$,
\begin{equation}\label{eq:sec2:41}
    \mathcal{L}_{\epsilon, \mathcal{M}}f = (\mathcal{L}_{\epsilon, m}f)_{m\in\mathcal{M}},\quad \mathcal{L}_{\epsilon, m}f = \frac{c_{m}}{\epsilon}\cdot\partial_{x}f_{m} +\frac{s_{m}}{\epsilon}\cdot\partial_{y}f_{m}+\frac{\sigma_{T}}{\epsilon^2}f_{m}-(\frac{\sigma_{T}}{\epsilon^{2}}-\sigma_{a})\sum_{p\in \mathcal{M}}\kappa_{m,p}f_{p}\omega_{p},
\end{equation}
where $\kappa_{m,p}$ approximates the scattering phase function $\kappa(\mathbf{u}_{m},\mathbf{u}_{p})$ \cite{chen2018uniformly}. 

The initial condition, Dirichlet boundary condition and the interface condition for $\bm{\psi}^{n}$ then becomes:
\begin{equation}\label{eq:sec2:27}
    \bm{\psi}^{0}(x,y)=\Big(\Psi^{0}(x,y,c_{m},s_{m})\Big)_{m\in\mathcal{M}},\quad (x, y)\in\Omega.
\end{equation}

\begin{equation}\label{eq:sec2:28}
    \bm{\psi}_{\Gamma^{-}}^{n}(x,y) \triangleq \big(\bm{\psi}_{m}^{n}(x,y)\big)_{\mathbf{n}_{x,y}\cdot \mathbf{u}_{m}<0}=\big(\Psi^{n}(x,y,c_{m}, s_{m})\big)_{\mathbf{n}_{x,y}\cdot \mathbf{u}_{m}<0}\triangleq\bm{\Psi}_{\Gamma^{-}}^{n}(x,y),\quad 1\leq n\leq N.
\end{equation}

\begin{equation}\label{eq:sec2:29}
    [\bm{\psi}^{n}(x,y)]|_{\alpha}=0,\quad 1\leq n\leq N.
\end{equation}

\subsection{Spatial discretization by ATFPS}
This subsection introduces the tailored finite point scheme (TFPS) \cite{han2014two}, applies it to our discretization framework, and then presents its adaptive extension, the ATFPS \cite{AdaptiveTFPS}. This adaptive variant exploits any low-rank structure present in the angular domain to compress it adaptively while preserving the underlying layer structure, thereby enhancing computational efficiency. 
\subsubsection{TFPS}
The TFPS employs the fundamental solutions of the RTE with constant coefficients as local basis functions, where the constants are taken to be the cell averages. These local basis functions are then stitched together to construct a globally defined approximate solution by enforcing continuity at the centers of interior cell interfaces and consistency with the prescribed boundary conditions at the centers of boundary cell interfaces. Notably, TFPS achieves uniform second-order accuracy in the mesh width $h$ with respect to the mean free path $\epsilon$ , even in the presence of boundary layers. In the following, we present the TFPS in detail.

To begin with, we uniformly partition the physical domain $\Omega = [0,1] \times [0,1]$ into an $I \times I$ grid with mesh size $h = 1/I$. The grid points are defined as $x_i = ih$ ($0 \leq i \leq I$) and $y_j = jh$ ($0 \leq j \leq I$), as illustrated in Figure \ref{fig:mesh}. Each cell $C_{i,j}$ is defined as the rectangle $[x_{i-1},x_{i}]\times[y_{j-1},y_{j}]$, with $1\leq i\leq I$ and $1\leq j\leq I$ serving as the row and column indices. The set of all cells is denoted as $\mathcal{C}=\{C_{i,j}\}_{1\leq i,j\leq I}$.
\begin{figure}[htbp]
    \centering
    \includegraphics[width=0.3\linewidth]{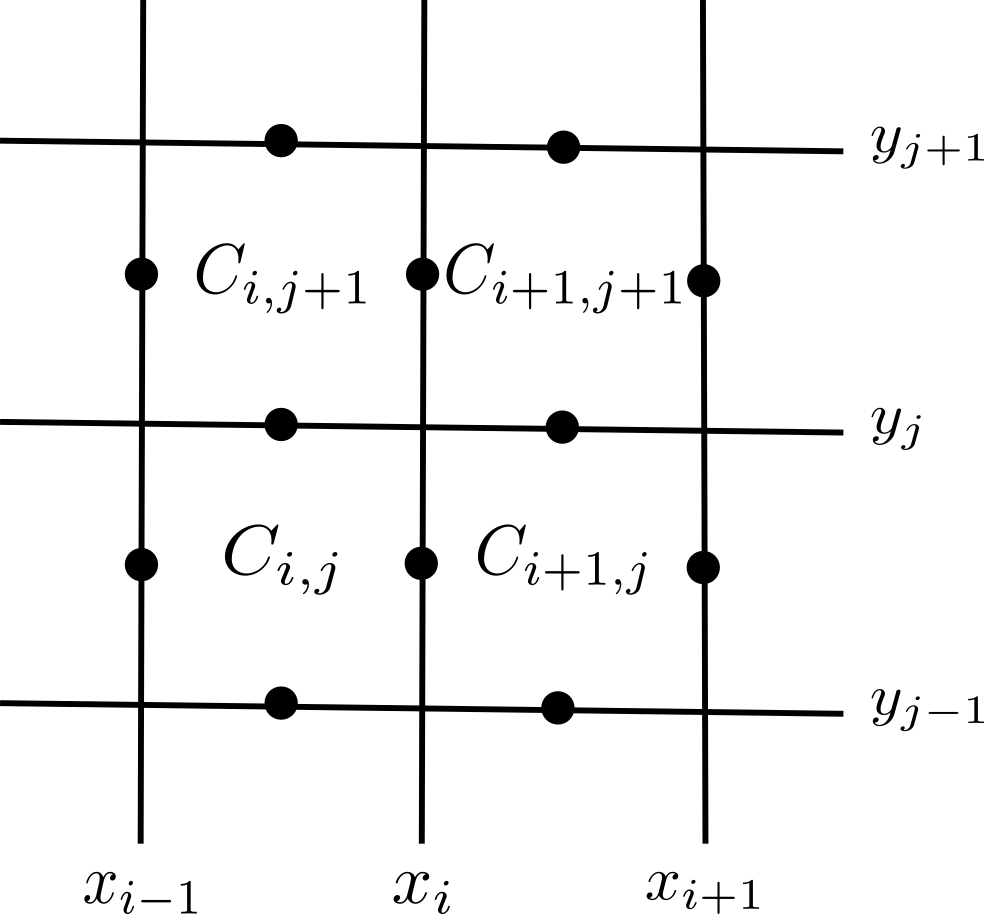}
    \caption{Mesh}
    \label{fig:mesh}
\end{figure}
The TFPS approximates the coefficients - namely, the total and absorption cross sections, as well as the source term -by their respective cell averages, which leads to the following scheme:
\begin{equation}\label{eq:sec2:20}
    \frac{\bar{\bm{\psi}}^{n}-\bar{\bm{\psi}}^{n-1}}{\Delta t}+\mathcal{L}_{\epsilon, \mathcal{M}, h}\frac{\bm{\psi}^{n}+\bm{\psi}^{n-1}}{2}=\bar{q}^{n-1/2},\quad \ 1\leq n\leq N,\ n\in\mathbb{Z},
\end{equation}
with 
\[
\bar{\bm{\psi}}^{n,s} = \sum_{C\in\mathcal{C}}\frac{\mathbf{1}_{C}}{|C|}\int_{C}\bm{\psi}^{n,s}dxdy,\quad \bar{\bm{\psi}}^{n-1} = \sum_{C\in\mathcal{C}}\frac{\mathbf{1}_{C}}{|C|}\int_{C}\bm{\psi}^{n-1}dxdy,\quad \bar{q}^{n-1/2} = \sum_{C\in\mathcal{C}}\frac{\mathbf{1}_{C}}{|C|}\int_{C}q^{n-1/2}dxdy.
\]
Here, $\mathbf{1}_{C}$ denotes the characteristic function of cell $C$, and $|C|$ is the area of cell $C$.
Based on the definition of $\mathcal{L}_{\epsilon,\mathcal{M}}$ in \eqref{eq:sec2:41}, the spatially discretized operator $\mathcal{L}_{\epsilon,\mathcal{M}, h}$ acts on a vector-valued function $f=(f_{m})_{m\in \mathcal{M}}$ as:
\[
\mathcal{L}_{\epsilon, \mathcal{M}, h}f = (\mathcal{L}_{\epsilon, m, h}f)_{m\in\mathcal{M}},\quad \mathcal{L}_{\epsilon, m, h}f = \frac{c_{m}}{\epsilon}\cdot\partial_{x}f_{m} +\frac{s_{m}}{\epsilon}\cdot\partial_{y}f_{m}+\frac{\sigma_{T}}{\epsilon^2}f_{m}-(\frac{\bar{\sigma}_{T}}{\epsilon^{2}}-\bar{\sigma}_{a})\sum_{p\in \mathcal{M}}\kappa_{m,p}f_{p}\omega_{p},
\]
where $\bar{\sigma}_{T}$ and $\bar{\sigma}_{a}$ are their respective cell averages:
\[
\bar{\sigma}_{T} = \sum_{C\in\mathcal{C}}\frac{\mathbf{1}_{C}}{|C|}\int_{C}\sigma_{T}dxdy,\quad \bar{\sigma}_{a} = \sum_{C\in\mathcal{C}}\frac{\mathbf{1}_{C}}{|C|}\int_{C}\sigma_{a}dxdy.
\]
In practice, equation \eqref{eq:sec2:20} is solved via a fixed-point iteration to guarantee convergence:
\begin{equation}\label{eq:sec2:21}
\begin{aligned}
    &\frac{\bar{\bm{\psi}}^{n, s-1}-\bar{\bm{\psi}}^{n-1}}{\Delta t}+\mathcal{L}_{\epsilon, \mathcal{M}, h}\frac{\bm{\psi}^{n,\mathrm{ds}}+\bm{\psi}^{n-1}}{2} = \bar{q}^{n-1/2}, \\
    &\bm{\psi}^{n,s} = (1-\Delta t)\bm{\psi}^{n, s-1} + \Delta t \bm{\psi}^{n,\mathrm{ds}}, \quad s \geq 1.
\end{aligned}
\end{equation}
Upon convergence, the final iterate $\bm{\psi}^{n,s}$ is accepted as the numerical solution $\bm{\psi}^n$ at time step $n$.

Then the discrete solution space for the TFPS is constructed as $\mathcal{F}+P_{0}(\mathcal{C},\mathcal{M})$, where $\mathcal{F}$ denotes the space spanned by local fundamental solutions and $P_{0}(\mathcal{C},\mathcal{M})$ represents a space of piecewise-constant special solutions. Specifically,
\begin{equation}\label{eq:sec2:26}
    \mathcal{F}=\{\sum_{C\in\mathcal{C}}\sum_{k\in \mathcal{V}}\alpha_{C}^{(k)}\phi_{C}^{k}|\forall \alpha_{C}^{(k)}\in\mathbb{R}\}, \quad P_{0}(\mathcal{C},\mathcal{M}) = \{\sum_{C\in\mathcal{C}}\mathbf{v}\cdot\mathbf{1}_{C}|\mathbf{v}\in\mathbb{R}^{4M}\}.
\end{equation}
Here, $\{\phi_{C}^{k}\}_{k\in\mathcal{V}}$ with $\mathcal{V}=\{1,2,\dots,8M\}$ denotes the set of local separable solutions to $\mathcal{L}_{\epsilon,\mathcal{M},h}f=0$ in cell $C$. Their explicit forms are provided in Appendix \ref{appendix:basis}. 

To be consistent with the discrete solution space, the boundary and continuity conditions for equation \eqref{eq:sec2:21} are imposed as follows:
\begin{equation}\label{eq:sec2:22}
    \bm{\psi}_{\Gamma^{-}}^{n,s}(\bar{\mathbf{x}}_{\mathfrak{i}})=\bm{\Psi}_{\Gamma^{-}}^{n}(\bar{\mathbf{x}}_{\mathfrak{i}}) ,\quad \quad \forall \mathfrak{i}\in \mathcal{I}_{\mathrm{b}}.
\end{equation}
\begin{equation}\label{eq:sec2:23}
    [\bm{\psi}^{n,s}(\bar{\mathbf{x}}_{\mathfrak{i}})]=0,\quad \forall \mathfrak{i}\in \mathcal{I}_{\mathrm{in}}.
\end{equation}
Here, the notation $\mathcal{I}$ represents the set of all cell interfaces, while $\mathcal{I}_{\mathrm{b}}$ and $\mathcal{I}_{\mathrm{in}}$ represent its restriction to the physical domain boundary and the physical domain interior, respectively. $\bar{\mathbf{x}}_{\mathfrak{i}}$ denotes the center of interface $\mathfrak{i}$ for any $\mathfrak{i}\in\mathcal{I}$. The symbol $[\cdot]$ denotes the jump in the function value at a given point across an interface. 

Note that in the computation of $\bm{\psi}^{n,s}$, its special solution component can be obtained directly from equation \eqref{eq:sec2:21}, since the operator $\mathcal{L}_{\epsilon, \mathcal{M}, h}$ maps all fundamental solutions to zero. The fundamental solution component are then uniquely determined by enforcing the continuity and boundary conditions specified in equations \eqref{eq:sec2:22} and \eqref{eq:sec2:23}.

\subsubsection{ATFPS}
Different from TFPS, the discrete solution space for the ATFPS solution with threshold $\delta$ for basis function selection is $\mathcal{F}_{\delta}+P_{0}(\mathcal{C},\mathcal{M})$, where the new fundamental solution space $\mathcal{F}_{\delta}$ is defined as:
\begin{equation}\label{eq:sec2:30}
    \mathcal{F}_{\delta}=\{\sum_{C\in\mathcal{C}}\sum_{k\in \mathcal{V}_{\delta, C}}\alpha_{C}^{(k)}\phi_{C}^{k}|\forall \alpha_{C}^{(k)}\in\mathbb{R}\}, \quad P_{0}(\mathcal{C},\mathcal{M}) = \{\sum_{C\in\mathcal{C}}\mathbf{v}\cdot\mathbf{1}_{C}|\mathbf{v}\in\mathbb{R}^{4M}\}.
\end{equation}
Here, $\mathcal{V}_{\delta, C}$ is a subset of $\mathcal{V}$ comprising indices of basis functions in cell $C$ that exhibit slow decay, rigorously defined as
\[
\mathcal{V}_{\delta, C} = \big\{k\in\mathcal{V}\big|\vert\phi_{C}^{k}(\bar{\mathbf{x}}_{C})\vert > \delta\big\}.
\]
where $\bar{\mathbf{x}}_{C}$ denotes the center point of cell $C$. The selection of basis functions is clearly governed by the threshold $\delta$ and the values of $\sigma_{T}$, $\sigma_{a}$, $\epsilon$, and $\kappa_{m,n}$ within cell $C$.

For convenience in the subsequent presentation, we introduce the following notations. The index set $\mathcal{V}_{\delta, C}$ can be decomposed as $\mathcal{V}_{\delta, C}=\cup_{\mathfrak{i}\in C}\mathcal{V}_{\delta, C, \mathfrak{i}}$, where $\mathcal{V}_{\delta, C,\mathfrak{i}}$ denotes the indices of the slow-decaying basis functions in cell $C$ that are centered on the interface $\mathfrak{i}$. Furthermore, the index sets $\mathcal{V}_{\delta,C,\mathfrak{i}}$ that share the same interface $\mathfrak{i}$ can be merged to form a new index set $\mathcal{V}_{\delta,\mathfrak{i}}=\bigcup_{C:\mathfrak{i}\cap C\ne\emptyset}\mathcal{V}_{\delta,C,\mathfrak{i}}$.

Using the same special solution as in the standard TFPS-which ensures that equation \eqref{eq:sec2:21} is satisfied-the fundamental solution component of the ATFPS solution $\bm{\psi}_{\delta}^{n,s}\in\mathcal{F}_{\delta}+P_{0}(\mathcal{C},\mathcal{M})$ can be determined by the following specially designed continuity and boundary conditions:
\begin{equation}\label{eq:sec3:31}
    \mathcal{P}_{\delta,\mathfrak{i}}(\bm{\psi}_{\Gamma^{-},\delta}^{n,s}(\mathbf{x}_{\mathfrak{i},\mathrm{mid}})-\bm{\Psi}_{\Gamma^{-}}^{n}(\mathbf{x}_{\mathfrak{i},\mathrm{mid}}))=0 ,\quad \quad \forall \mathfrak{i}\in \mathcal{I}_{\mathrm{b}},\quad 1\leq n\leq N.
\end{equation}
\begin{equation}\label{eq:sec3:32}
    \mathcal{P}_{\delta,\mathfrak{i}}([\bm{\psi}^{n,s}(\mathbf{x}_{\mathfrak{i},\mathrm{mid}})])=0,\quad \forall \mathfrak{i}\in \mathcal{I}_{\mathrm{in}},\quad 1\leq n\leq N.
\end{equation}

Here, the projection operator $\mathcal{P}_{\delta,\mathfrak{i}}$ maps the $2M$-dimensional vector space (when $\mathfrak{i}\in\mathcal{I}_{\mathrm{b}}$) or the $4M$-dimensional vector space (when $\mathfrak{i}\in\mathcal{I}_{\mathrm{in}}$) onto a $\vert\mathcal{V}_{\delta,\mathfrak{i}}\vert$-dimensional vector, whose entries correspond to the components of the original vector along the important velocity modes. Its specific definition is provided in the Appendix \ref{appendix:projection_operator}.

When an inherent low-rank structure is present in the angular domain, the ATFPS exploits this property by retaining only a reduced set of local basis functions. Specifically, the total number of retained basis functions, $\sum_{C \in \mathcal{C}} |\mathcal{V}_{\delta, C}|$, can be significantly smaller than the original count $\sum_{C \in \mathcal{C}} |\mathcal{V}|$. This reduction leads to a compressed linear system, thereby enhancing computational efficiency.

However, since the ATFPS retains only the slowly decaying basis functions, its raw approximation—denoted by $\bm{\psi}_{\delta}^{n,s}$—captures only the smooth component of the exact solution to the system defined by equations \eqref{eq:sec2:21}-\eqref{eq:sec2:23}. To recover the fine-scale features associated with boundary and interface layers (i.e., the information carried by the discarded fast-decaying basis functions), we apply the layer reconstruction technique introduced in \cite{AdaptiveTFPS}.

The resulting reconstructed solution is denoted by $\tilde{\bm{\psi}}_{\delta}^{n,s,*}$, where the superscript asterisk indicates that layer information has been restored. As established in \cite{han2014two, AdaptiveTFPS}, this reconstructed solution $\tilde{\bm{\psi}}_{\delta}^{*}$ achieves uniform second-order accuracy in the mesh size $h$ and first-order accuracy in the compression parameter $\delta$ for approximating the uncompressed solution $\bm{\psi}$ over the entire physical domain—uniformly with respect to the mean free path $\epsilon$.

\begin{remark}
    As the time step $\Delta t$ decreases, the convergence of the iteration \eqref{eq:sec2:21} gradually slows down. Therefore, when $\Delta t$ is sufficiently small, we try to use the following iteration instead:
    \[
    \frac{\bar{\bm{\psi}}^{n, s}-\bar{\bm{\psi}}^{n-1}}{\Delta t}+\mathcal{L}_{\epsilon, \mathcal{M}, h}\frac{\bm{\psi}^{n,s-1}+\bm{\psi}^{n-1}}{2}=\bar{q}^{n-1/2},\quad \ 1\leq n\leq N,\ n\in\mathbb{Z},\quad s\geq 1,\ s \in\mathbb{Z}.
    \]
However, since the rapidly decaying basis functions have a non-negligible contribution to the cell average of $\bm{\psi}^{n, s}$, we cannot solely operate on the slow-decaying basis functions, which causes the aforementioned angular compression to fail. Consequently, we replace all cell averages in \eqref{eq:sec2:21} with cell center values as follows:
\begin{equation}\label{eq:sec2:24}
    \frac{\bm{\psi}_{c}^{n, s}-\bm{\psi}_{c}^{n-1}}{\Delta t}+\mathcal{L}_{\epsilon, \mathcal{M}, h}\frac{\bm{\psi}^{n,s-1}+\bm{\psi}^{n-1}}{2}=q_{c}^{n-1/2},\quad \ 1\leq n\leq N,\ n\in\mathbb{Z},\quad s\geq 1,\ s \in\mathbb{Z}.
\end{equation}
Here, the subscript $c$ denotes the discrete cell center values. 

\end{remark}

\section{Recursive Skeleton (RS)}
\label{sec:RS}
As observed, the dominant computational cost arises from solving the coupled system of equations \eqref{eq:sec2:21} (or \eqref{eq:sec2:24}), \eqref{eq:sec3:31}, and \eqref{eq:sec3:32}. This system essentially determines a solution $f\in\mathcal{F}_{\delta}$ from its projected values at boundary interface centers, $\mathcal{P}_{\delta,\mathfrak{i}}^{k}f(\bar{\mathbf{x}}_{\mathfrak{i}})$ for all $\mathfrak{i}\in \mathcal{I}_{\mathrm{b}}$, $k\in\mathcal{V}_{\delta,\mathfrak{i}}$, and its projected jump values at interior interface centers, $\mathcal{P}_{\delta,\mathfrak{i}}^{k}[f(\bar{\mathbf{x}}_{\mathfrak{i}})]$ for all $\mathfrak{i}\in \mathcal{I}_{\mathrm{in}}$, $k\in\mathcal{V}_{\delta,\mathfrak{i}}$. We term this correspondence the solution map. Since this map is used repeatedly for each time step ($n$) and iteration step ($s$) in \eqref{eq:sec2:21} or \eqref{eq:sec2:24}, our next objective is to derive a multilevel decomposition for its explicit expression.

The underlying technique, known as the recursive skeleton method (RSM) \cite{ho2012fast}, has been successfully applied to integral formulations of steady‑state radiative transport equations \cite{fan2019fast, ren2019fast}. In this work, we extend this approach to our adaptively compressed system and reinterpret it within a new framework.

In this section, we begin by illustrating the multilevel decomposition of the fundamental solution space $\mathcal{F}_{\delta}$. The space is recursively decomposed as $\mathcal{F}_{\delta}=\mathcal{F}_{\delta}^{(0)}=\sum_{l=1}^{(L)}\mathcal{G}_{\delta}^{(l)}+\mathcal{F}_{\delta}^{(L)}$, following the relation $\mathcal{F}_{\delta}^{(l)}=\mathcal{F}_{\delta}^{(l+1)}+\mathcal{G}_{\delta}^{(l+1)}$. 
Finding a solution $f\in\mathcal{F}_{\delta}$ is then transformed into a sequence of projections: first finding its projection $g_{\delta}^{(l)}$ onto $\mathcal{G}_{\delta}^{(l)}$ from $l=1$ to $l=L$, and then finding its projection $f^{(L)}$ onto the coarsest space $\mathcal{F}_{\delta}^{(L)}$. Subsequently, $f$ is reconstructed by summing these components: $f = f^{(L)} + \sum_{l=1}^{(L)} g_{\delta}^{(l)}$.
Representing this process using linear operators yields the multilevel decomposition of the solution map. Finally, we present an analysis of the computational cost for assembling this multilevel decomposition and the storage requirements necessary for storing the decomposition.

\subsection{Notation}
We begin by introducing notation that will be used in subsequent sections. 
\begin{itemize}
    \item $L$: the total number of levels in the physical mesh hierarchy.
    \item $C^{(l)}$ for $0\leq l\leq L$: the set of all physical cells at mesh level $l$, as illustrated in Figure~\ref{fig:MultilevelMesh}.
    \item $I^{(l)}$ for $0\leq l\leq L$: the number of cells along each spatial axis at level $l$. Clearly, $I^{(l)} = \frac{I}{2^{l}}$, where $I$ denotes the number of cells on the finest mesh (as defined in the previous section).
    \item $\mathcal{I}^{(l)}$ for $0\leq l\leq L$: the set of all cell interfaces at level $l$, partitioned as $\mathcal{I}^{(l)}=\mathcal{I}_{\mathrm{b}}^{(l)}\cup\mathcal{I}_{\mathrm{in}}^{(l)}$, where $\mathcal{I}_{\mathrm{b}}^{(l)}$ and $\mathcal{I}_{\mathrm{in}}^{(l)}$ denote the sets of boundary and interior interfaces at level $l$, respectively. Note that under our multilevel construction, we have the nested relations $\mathcal{I}^{(l_{1})}\subset \mathcal{I}^{(l_{2})}$, $\mathcal{I}_{\mathrm{in}}^{(l_{1})}\subset \mathcal{I}_{\mathrm{in}}^{(l_{2})}$, $\mathcal{I}_{\mathrm{b}}^{(l_{1})}\subset \mathcal{I}_{\mathrm{b}}^{(l_{2})}$ for $l_{1}<l_{2}$.
    \item $\mathcal{F}_{\delta}^{(l)}$ for $0\leq l\leq L$ and $\mathcal{G}_{\delta}^{(l)}$ for $1\leq l\leq L$: the solution spaces for each level of mesh as defined in \eqref{eq:sec4:10} and \eqref{eq:sec4:11}.
    \item $\{\phi_{C,\mathfrak{i}}^{(l),k}\}_{C\in\mathcal{C}^{(l)},\mathfrak{i}\in\mathcal{I}^{(l)}\cap C,k\in\mathcal{V}_{\delta,C,\mathfrak{i}}}$ for $0\leq l\leq L$: the localized basis functions for $\mathcal{F}_{\delta}^{(l)}$, which is defined in \eqref{eq:sec3:2}.
    \item $\{\phi_{C,\mathfrak{i}}^{(l),k}\}_{C\in\mathcal{C}^{(l)},\mathfrak{i}\in(\mathcal{I}^{(l-1)}\setminus\mathcal{I}^{(l)})\cap C,k\in\mathcal{V}_{\delta,C,\mathfrak{i}}}$ for $1\leq l\leq L$: the localized basis functions for $\mathcal{G}_{\delta}^{(l)}$, which is defined in \eqref{eq:sec3:3}.
    \item $B_{\delta}^{(l)}$ for $0\leq l\leq L$, and $P_{\delta}^{(l)}$, $Q_{\delta}^{(l)}$ and $R_{\delta}^{(l)}$ for $0\leq l\leq L-1$: the key components of the multilevel decomposition of the solution operator, as defined in subsection \ref{subsec:multilevel_decomposition_operator}.
\end{itemize}

\subsection{Construction of nested solution spaces and the explicit expressions of their basis}
In this subsection, we first construct a hierarchy of cells and cell interfaces, and then define the corresponding solution spaces based on the nesting relationships. For computation efficiency, we also explicitly construct the basis functions for the solution spaces at each mesh level and present their detailed properties.

\subsubsection{Nested cells and cell interfaces}
Denote the total number of mesh levels by $L+1$, where level $0$ corresponds to the finest mesh and level $L$ to the coarsest. For each mesh level $l$ ($0 \leq l \leq L$), let $I^{(l)}$ denote the number of cells along each coordinate axis, and assume that the coarsest level satisfies $I^{(L)} = O(1)$.

Let $\mathcal{C}^{(l)}$ be the set of all cells at level $l$, and let $\mathcal{I}^{(l)}$ denote the set of all cell interfaces at that level, which is partitioned into boundary interfaces $\mathcal{I}_{\mathrm{b}}^{(l)}$ and interior interfaces $\mathcal{I}_{\mathrm{in}}^{(l)}$, i.e., $\mathcal{I}^{(l)} = \mathcal{I}_{\mathrm{b}}^{(l)} \cup \mathcal{I}_{\mathrm{in}}^{(l)}$. At the finest level ($l = 0$), we simplify the notation by writing $I^{(0)}$, $\mathcal{C}^{(0)}$, $\mathcal{I}^{(0)}$, $\mathcal{I}_{\mathrm{b}}^{(0)}$, $\mathcal{I}_{\mathrm{in}}^{(0)}$ as $I$, $\mathcal{C}$, $\mathcal{I}$, $\mathcal{I}_{\mathrm{b}}$, $\mathcal{I}_{\mathrm{in}}$, respectively, 
which is consistent with the notation used in the previous section. An illustration of the cells, interfaces, and grid points across different mesh levels is provided in Figure \ref{fig:MultilevelMesh}.

\begin{figure}[htbp]
    \centering
    \begin{minipage}{0.25\textwidth}
        \centering
        \includegraphics[width=\linewidth]{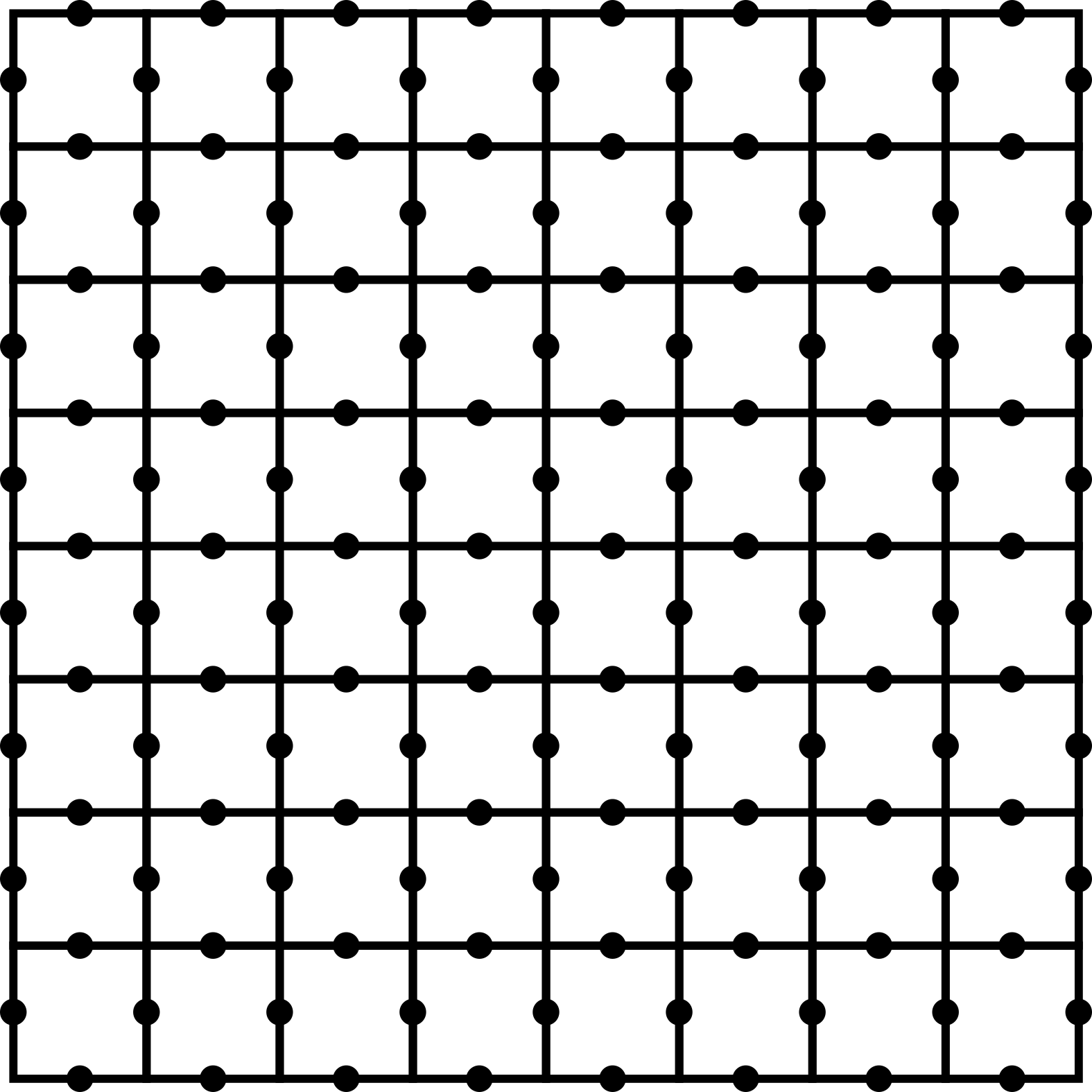}
        \small (a)
    \end{minipage}
    \begin{minipage}{0.25\textwidth}
        \centering
        \includegraphics[width=\linewidth]{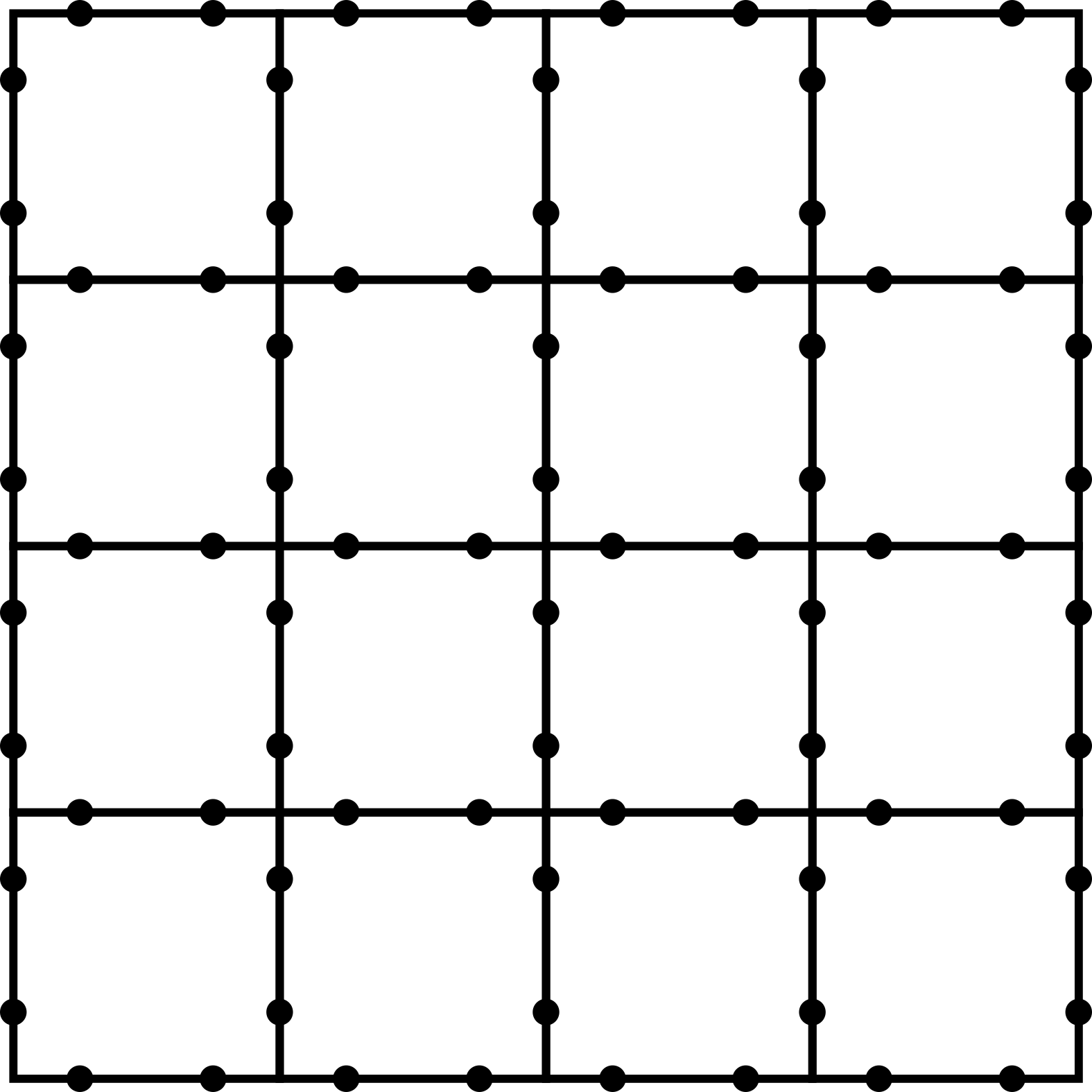}
        \small (b)
    \end{minipage}
    \begin{minipage}{0.25\textwidth}
        \centering
        \includegraphics[width=\linewidth]{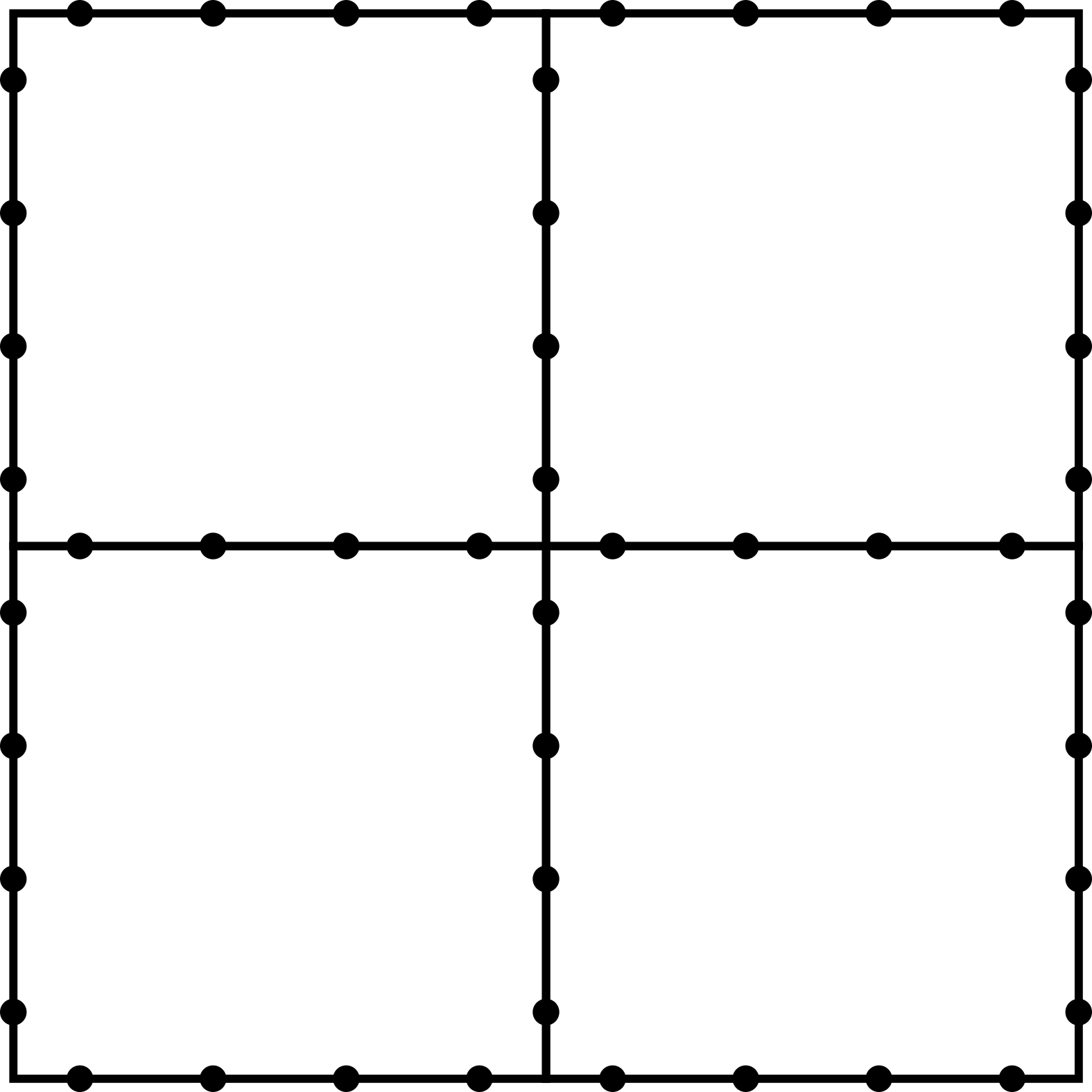}
        \small (c)
    \end{minipage}
    \caption{Cells, interfaces and grid points on each level of mesh. Here we take $I=8$ and $L=2$ for example. (a) level 0; (b) level 1; (c) level 2.}
    \label{fig:MultilevelMesh}
\end{figure}

\subsubsection{Nested solution spaces $\mathcal{F}_{\delta}^{(l)} (0\leq l\leq L)$ and $\mathcal{G}_{\delta}^{(l)} (1\leq l\leq L)$}
Based on the above notations of nested cells and interfaces, we introduce the following solution spaces:
\begin{equation}\label{eq:sec4:10}
    \mathcal{F}_{\delta}^{(0)}=\mathcal{F}_{\delta},\quad \mathcal{F}_{\delta}^{(l)}\triangleq\Big\{f\in\mathcal{F}_{\delta}^{(l-1)}\Big|\mathcal{P}_{\delta,\mathfrak{i}}^{k}[f(\bar{\mathbf{x}}_{\mathfrak{i}})]=0,\forall \mathfrak{i}\in\mathcal{I}_{\mathrm{in}}^{(l-1)}\setminus\mathcal{I}_{\mathrm{in}}^{(l)},\forall k\in\mathcal{V}_{\delta,\mathfrak{i}}\Big\},\quad 1\leq l\leq L;
\end{equation}
\begin{equation}\label{eq:sec4:11}
    \mathcal{G}_{\delta}^{(l)} = \Big\{f\in\mathcal{F}_{\delta}^{(l-1)}\Big|\mathcal{P}_{\delta,\mathfrak{i}}^{k}(f(\bar{\mathbf{x}}_{\mathfrak{i}})|_{C})=0,\forall C\in\mathcal{C}^{(l)},\ \forall\mathfrak{i}\in\mathcal{I}^{(l)}\cap C,\ \forall k\in\mathcal{V}_{\delta,C,\mathfrak{i}}\Big\},\quad 1\leq l \leq L.
\end{equation}
Here, for any cell $C\in\mathcal{C}^{(l)}$ with $0\leq l\leq L$, we extend the notation $\mathcal{V}_{\delta,C,\mathfrak{i}}$ for cells across different mesh levels as follows: $\mathcal{V}_{\delta,C,\mathfrak{i}}=\bigcup_{C_{0}\subset C, C_{0}\in\mathcal{C}}\mathcal{V}_{\delta,C_{0},\mathfrak{i}}$.
The relationship between $\mathcal{F}_{\delta}^{(l)}$ and $\mathcal{G}_{\delta}^{(l)}$ can then be established in Lemma \ref{lemma:FG}.

\begin{lemma}\label{lemma:FG}
    \[\mathcal{F}_{\delta}^{(l-1)}=\mathcal{F}_{\delta}^{(l)}\oplus\mathcal{G}_{\delta}^{(l)},\quad 1\leq l\leq L.\]
\end{lemma}
The proof is shown in Appendix \ref{sec:proof_nested_solution_space}.

Using Lemma \ref{lemma:FG}, we obtain the multilevel decomposition of the fundamental solution space $\mathcal{F}_{\delta}$ as follows:
\begin{equation}\label{eq:sec4:23}
    \mathcal{F}_{\delta}=\mathcal{F}_{\delta}^{(0)}=\mathcal{G}_{\delta}^{(1)}\oplus\mathcal{F}_{\delta}^{(1)}=\mathcal{G}_{\delta}^{(1)}\oplus\mathcal{G}_{\delta}^{(2)}\oplus\mathcal{F}_{\delta}^{(2)}=\dots = (\oplus_{l=1}^{L}\mathcal{G}_{\delta}^{(l)})\oplus\mathcal{F}_{\delta}^{(L)}.
\end{equation}

\subsubsection{Basis functions of $\mathcal{F}_{\delta}^{(l)} (0\leq l\leq L)$ and $\mathcal{G}_{\delta}^{(l)} (1\leq l\leq L)$}
Based on the definitions of $\mathcal{F}_{\delta}^{(l)}$ and $\mathcal{G}_{\delta}^{(l)}$ given in \eqref{eq:sec4:10} and \eqref{eq:sec4:11}, we now define their respective basis functions.

For $0 \leq l \leq L$, the basis functions of $\mathcal{F}_{\delta}^{(l)}$ are constructed as follows. For any cell $C_0 \in \mathcal{C}^{(l)}$, interface $\mathfrak{i}_0 \in \mathcal{I}^{(l)} \cap C_0$, and index $k_0 \in \mathcal{V}_{\delta, C_0, \mathfrak{i}_0}$, the function $\phi_{C_0, \mathfrak{i}_0}^{(l), k_0}$ is defined by
\begin{equation}\label{eq:sec3:2}
    \begin{aligned}
        &\phi_{C_{0},\mathfrak{i}_{0}}^{(l),k_{0}} \in\mathcal{F}_{\delta}^{(l)}\\
        s.t.\quad & \mathcal{P}_{\delta,\mathfrak{i}_{0}}^{k_{0}}(\phi_{C_{0},\mathfrak{i}_{0}}^{(l),k_{0}}(\bar{\mathbf{x}}_{\mathfrak{i}_{0}})|_{C_{0}})=1;\\
        &\mathcal{P}_{\delta,\mathfrak{i}}^{k}(\phi_{C_{0},\mathfrak{i}_{0}}^{(l),k_{0}}(\bar{\mathbf{x}}_{\mathfrak{i}})|_{C})=0,\quad \forall\ C\in\mathcal{C}^{(l)},\ \mathfrak{i}\in \mathcal{I}^{(l)}\cap C, \  k\in\mathcal{V}_{\delta,C,\mathfrak{i}},\ (C,\mathfrak{i},k)\ne (C_{0},\mathfrak{i}_{0},k_{0}).\\
    \end{aligned}
\end{equation}

For $1 \leq l \leq L$, the basis functions of $\mathcal{G}_{\delta}^{(l)}$ are defined as follows. For any $C_0 \in \mathcal{C}^{(l)}$, interface $\mathfrak{i}_0 \in (\mathcal{I}^{(l-1)} \setminus \mathcal{I}^{(l)}) \cap C_0$, and index $k_0 \in \mathcal{V}_{\delta, C_0, \mathfrak{i}_0}$, the function $\phi_{C_0, \mathfrak{i}_0}^{(l), k_0}$ satisfies
\begin{equation}\label{eq:sec3:3}
    \begin{aligned}
        &\phi_{C_{0},\mathfrak{i}_{0}}^{(l),k_{0}} \in\mathcal{G}_{\delta}^{(l)}\\
        s.t.\quad & \mathcal{P}_{\delta,\mathfrak{i}_{0}}^{k_{0}}[\phi_{C_{0},\mathfrak{i}_{0}}^{(l),k_{0}}(\mathbf{x}_{\mathfrak{i},\mathrm{mid}})]=1; ;\\
         \quad &\mathcal{P}_{\delta,\mathfrak{i}}^{k}[\phi_{C_{0},\mathfrak{i}_{0}}^{(l),k_{0}}(\mathbf{x}_{\mathfrak{i},\mathrm{mid}})]=0\quad \forall,\  \mathfrak{i}\in \mathcal{I}^{(l-1)}\setminus\mathcal{I}^{(l)},\  k\in\mathcal{V}_{\delta,\mathfrak{i}},\ (\mathfrak{i},k)\ne(\mathfrak{i}_{0},k_{0}).
    \end{aligned}
\end{equation}

These basis functions are clearly linearly independent. Moreover, their total number matches the dimensions of $\mathcal{F}_{\delta}^{(l)}$ and $\mathcal{G}_{\delta}^{(l)}$, as computed in equations \eqref{eq:sec3:4} and \eqref{eq:sec3:5}. Consequently, they form complete bases for $\mathcal{F}_{\delta}^{(l)}$ and $\mathcal{G}_{\delta}^{(l)}$, respectively, yielding the following characterizations:
\begin{equation}\label{eq:sec3:6}
    \begin{aligned}
        \mathcal{F}_{\delta}^{(l)} = \{\sum_{C\in\mathcal{C}^{(l)}}\sum_{\mathfrak{i}\in\mathcal{I}^{(l)}\cap C}\sum_{k\in\mathcal{V}_{\delta,C,\mathfrak{i}}}\alpha_{C,\mathfrak{i}}^{(l),k}\phi_{C,\mathfrak{i}}^{(l),k}|\alpha_{C,\mathfrak{i}}^{(l),k}\in\mathbb{R}\},\quad 0\leq l\leq L;\\
        \mathcal{G}_{\delta}^{(l)} = \{\sum_{C\in\mathcal{C}^{(l)}}\sum_{\mathfrak{i}\in(\mathcal{I}^{(l-1)}\setminus\mathcal{I}^{(l)})\cap C}\sum_{k\in\mathcal{V}_{\delta,\mathfrak{i}}}\alpha_{C,\mathfrak{i}}^{(l),k}\phi_{C,\mathfrak{i}}^{(l),k}|\alpha_{C,\mathfrak{i}}^{(l),k}\in\mathbb{R}\},\quad 1\leq l\leq L;
    \end{aligned}
\end{equation}

\subsubsection{Properties of the basis functions}
As a consequence of Definitions \eqref{eq:sec3:2} and \eqref{eq:sec3:3}, we obtain the following Lemma \ref{lemma:localization}.

\begin{lemma}\label{lemma:localization}
The basis functions of $\mathcal{F}_{\delta}^{(l)}$ for $0\leq l\leq L$ and $\mathcal{G}_{\delta}^{(l)}$ for $1\leq l\leq L$, defined in \eqref{eq:sec3:2}, and \eqref{eq:sec3:3} with the form $\phi_{C_{0},\mathfrak{i}_{0}}^{(l),k_{0}}$, are supported within the cell $C_{0}$.
\end{lemma}
The proof of Lemma \ref{lemma:localization} is shown in Appendix \ref{sec:proof_localization}.

Based on the expressions of $\mathcal{F}_{\delta}^{(l)}$ and $\mathcal{G}_{\delta}^{(l)}$ in \eqref{eq:sec3:6}, together with Lemma  \ref{lemma:FG} and \ref{lemma:localization}, we obtain the following property. The proof is straightforward and is therefore omitted.

\begin{lemma}\label{lemma:expression}
   For any $g\in\mathcal{G}_{\delta}^{(l)}$ $(1\leq l\leq L)$, it holds that:
   \begin{equation}\label{eq:sec3:8}
       g = \sum_{C\in\mathcal{C}^{(l)}}\sum_{\mathfrak{i}\in(\mathcal{I}^{(l-1)}\setminus\mathcal{I}^{(l)})\cap C}\sum_{k\in\mathcal{V}_{\delta,\mathfrak{i}}}\mathcal{P}_{\delta,\mathfrak{i}}^{k}\big([g(\bar{\mathbf{x}}_{\mathfrak{i}})]\big)\phi_{C,\mathfrak{i}}^{(l),k}=\sum_{\mathfrak{i}\in\mathcal{I}^{(l-1)}\setminus\mathcal{I}^{(l)}}\sum_{k\in\mathcal{V}_{\delta,\mathfrak{i}}}\mathcal{P}_{\delta,\mathfrak{i}}^{k}\big([g(\bar{\mathbf{x}}_{\mathfrak{i}})]\big)\phi_{C,\mathfrak{i}}^{(l),k}.
   \end{equation}
   Similarly, for any $f\in\mathcal{F}_{\delta}^{(l)}$ $(0\leq l\leq L)$, the function can be expressed as:
   \begin{equation}\label{eq:sec3:7}
       f = \sum_{C\in\mathcal{C}^{(l)}}\sum_{\mathfrak{i}\in\mathcal{I}^{(l)}\cap C}\sum_{k\in\mathcal{V}_{\delta,C,\mathfrak{i}}}\mathcal{P}_{\delta,\mathfrak{i}}^{k}\big(f(\bar{\mathbf{x}}_{\mathfrak{i}})|_{C}\big)\phi_{C,\mathfrak{i}}^{(l),k}.
   \end{equation}
\end{lemma}

\subsection{Multilevel decomposition of the solution operator}
\label{subsec:multilevel_decomposition_operator}
For notational convenience, we introduce the following vectors associated with any function $f \in \mathcal{F}_{\delta}$. First, for any subset of cells $\mathcal{C}_{\mathrm{sub}} \subseteq \mathcal{C}^{(l)}$ for some $l$, define the projection vector
\[
\mathbf{f}_{\mathcal{C}_{\mathrm{sub}}} = \Big(\mathcal{P}_{\delta,\mathfrak{i}}^{k}\big(f(\bar{\mathbf{x}}_{\mathfrak{i}})|_{C}\big)\Big)_{C\in\mathcal{C}_{\mathrm{sub}},\mathfrak{i}\in\mathcal{I}\cap \partial C,k\in\mathcal{V}_{\delta,C,\mathfrak{i}}}.
\]
Its boundary and interior components are given by
\[\mathbf{f}_{\mathcal{C}_{\mathrm{sub}},\mathrm{b}} = \Big(\mathcal{P}_{\delta,\mathfrak{i}}^{k}\big(f(\bar{\mathbf{x}}_{\mathfrak{i}})|_{C}\big)\Big)_{C\in\mathcal{C}_{\mathrm{sub}},\mathfrak{i}\in\mathcal{I}_{\mathrm{b}}\cap \partial C,k\in\mathcal{V}_{\delta,C,\mathfrak{i}}},\quad \mathbf{f}_{\mathcal{C}_{\mathrm{sub}},\mathrm{in}} = \Big(\mathcal{P}_{\delta,\mathfrak{i}}^{k}\big(f(\bar{\mathbf{x}}_{\mathfrak{i}})|_{C}\big)\Big)_{C\in\mathcal{C}_{\mathrm{sub}},\mathfrak{i}\in\mathcal{I}_{\mathrm{in}}\cap \partial C,k\in\mathcal{V}_{\delta,C,\mathfrak{i}}}
\]

Moreover, for any subset $\mathcal{I}_{\mathrm{sub}} \subseteq \mathcal{I}$, we define the jump vector 
\[[\mathbf{f}]_{\mathcal{I}_{\mathrm{sub}}} = \Big(\mathcal{P}_{\delta,\mathfrak{i}}^{k}\big([f(\bar{\mathbf{x}}_{\mathfrak{i}})]\big)\Big)_{\mathfrak{i}\in\mathcal{I}_{\mathrm{sub}},k\in\mathcal{V}_{\delta,\mathfrak{i}}}.\]

These quantities enable the definition of the discrete operators that constitute the key components of the multilevel decomposition:
\begin{itemize}
\item $B_{\delta}^{(l)}$: for any $f\in\mathcal{F}_{\delta}^{(l)}$, the operator $B_{\delta}^{(l)}$ maps $\mathbf{f}_{\mathcal{C}^{(l)}}$ to $(\mathbf{f}_{\mathcal{C}^{(l)},\mathrm{b}}^{T}, [\mathbf{f}]_{\mathcal{I}_{\mathrm{in}}^{(l)}}^{T})^{T}$. Numerical results demonstrate that this operator is a bijection.
\item $P_{\delta}^{(l)}$: for any $f\in\mathcal{F}_{\delta}^{(l+1)}$, $P_{\delta}^{(l)}$ maps $\mathbf{f}_{\mathcal{C}^{(l+1)}}$ to $\mathbf{f}_{\mathcal{C}^{(l)}}$.
\item $Q_{\delta}^{(l)}$: for any $f\in\mathcal{G}_{\delta}^{(l+1)}$, $Q_{\delta}^{(l)}$ maps $[\mathbf{f}]_{\mathcal{I}^{(l)}\setminus\mathcal{I}^{(l+1)}}$ to $\mathbf{f}_{\mathcal{C}^{(l)}}$.
\item $R_{\delta}^{(l)}$ and $\check{R}_{\delta}^{(l)}$: for any $f\in\mathcal{F}_{\delta}^{(l)}$, these operators map $(\mathbf{f}_{\mathcal{C}^{(l)},\mathrm{b}}^{T}, [\mathbf{f}]_{\mathcal{I}_{\mathrm{in}}^{(l)}}^{T})^{T}$ to $(\mathbf{f}_{\mathcal{C}^{(l+1)},\mathrm{b}}^{T}, [\mathbf{f}]_{\mathcal{I}_{\mathrm{in}}^{(l+1)}}^{T})^{T}$ and $[\mathbf{f}]_{\mathcal{I}_{\mathrm{in}}^{(l)}\setminus\mathcal{I}_{\mathrm{in}}^{(l+1)}}$, respectively.
\end{itemize}
Then we have the following lemma.
\begin{lemma}\label{lemma:twoleveldecomposition}
$$B_{\delta}^{(l),-1} = P_{\delta}^{(l)}\mathcal{B}_{\delta}^{(l+1),-1}(R_{\delta}^{(l)} - B_{\delta}^{(l+1)}Q_{\delta}^{(l)}\check{R}_{\delta}^{(l)}) + Q_{\delta}^{(l)}\check{R}_{\delta}^{(l)}.$$
\end{lemma}
The proof of Lemma \ref{lemma:twoleveldecomposition} is shown in Appendix \ref{append:twoleveldecomposition}.

Based on Lemma \ref{lemma:twoleveldecomposition}, we can recursively obtain the multilevel decomposition of $B_{\delta}^{-1}$. By definition, this decomposition yields the desired solution map.

\subsection{Computation and storage cost}
\label{subsec:computation_storage_cost}

In the following subsection, we assume for simplicity that the true angular rank is uniformly equal to $r$, i.e., $|\mathcal{V}_{\delta,C}| = 4r$ for all $C \in \mathcal{C}$. Consequently, we have
\[
|\mathcal{V}_{\delta,C,\mathfrak{i}}| = r\,2^{l}, \qquad 
|\mathcal{V}_{\delta,C}| = 4r\,2^{l}, 
\qquad \forall\, C \in \mathcal{C}^{(l)}.
\]
The following Lemmas \ref{lemma:offline_computation_cost}, \ref{lemma:online_computation_cost} and \ref{lemma:storage_cost} demonstrate the advantageous computational and storage costs of the proposed method, particularly in the presence of a low-rank structure $(r < O(M))$.

\begin{lemma}\label{lemma:offline_computation_cost}
     The derivation of an explicit multilevel decomposition of the solution operator $B_{\delta}^{-1}$ incurs a total computational cost of $O(r^{3} I^{3})$. 
\end{lemma}
\begin{proof}
As analyzed in the previous subsections, the dominant computational cost arises from deriving the explicit expressions of the basis functions at different mesh levels. The matrices $B_{\delta}^{(l)}$, $P_{\delta}^{(l)}$, $Q_{\delta}^{(l)}$, $R_{\delta}^{(l)}$, and $\check{R}_{\delta}^{(l)}$ encode information about these basis functions and can be assembled at negligible additional cost once the basis functions are known. Hence, we focus only on the cost of constructing the basis functions across mesh levels.

The locality of the basis functions in $\mathcal{F}_{\delta}^{(l)}$ and $\mathcal{G}_{\delta}^{(l)}$ greatly simplifies the computation of their explicit expressions. For instance, for $\mathcal{F}_{\delta}^{(l)}$ with $1 \leq l \leq L$, equation \eqref{eq:sec3:2} can be simplified as follows (by substituting \eqref{eq:RSF_basis_expansion} into \eqref{eq:sec3:1}, see Appendix \ref{sec:proof_localization}):
\begin{equation}\label{eq:sec3:13}
    \begin{aligned}
        &\phi_{C_{0},\mathfrak{i}_{0}}^{(l),k_{0}} = \sum_{C\in\mathcal{C}^{(l-1)}\cap C_{0}}\sum_{\mathfrak{i}\in\mathcal{I}^{(l-1)}\cap C}\sum_{k\in\mathcal{V}_{\delta,C,\mathfrak{i}}}\alpha_{C,\mathfrak{i}}^{(l-1),k}\phi_{C,\mathfrak{i}}^{(l-1),k}\\
        s.t.\quad & \mathcal{P}_{\delta,\mathfrak{i}_{0}}^{k_{0}}(\phi_{C_{0},\mathfrak{i}_{0}}^{(l),k_{0}}(\bar{\mathbf{x}}_{\mathfrak{i}_{0}}))=1;\\
        &\mathcal{P}_{\delta,\mathfrak{i}}^{k}(\phi_{C_{0},\mathfrak{i}_{0}}^{(l),k_{0}}(\bar{\mathbf{x}}_{\mathfrak{i}}))=0,\quad \forall\ \mathfrak{i}\in \mathcal{I}^{(l)}\cap C_{0}, \  k\in\mathcal{V}_{\delta,C_{0},\mathfrak{i}},\ (\mathfrak{i},k)\ne (\mathfrak{i}_{0},k_{0});\\
        & \mathcal{P}_{\delta,\mathfrak{i}}^{k}[\phi_{C_{0},\mathfrak{i}_{0}}^{(l),k_{0}}(\bar{\mathbf{x}}_{\mathfrak{i}})]=0,\quad \forall \mathfrak{i}\in (\mathcal{I}^{(l-1)}\setminus\mathcal{I}^{(l)})\cap C_{0},\quad \forall k\in\mathcal{V}_{\delta,\mathfrak{i}}.
    \end{aligned}
\end{equation}
The number of degrees of freedom in \eqref{eq:sec3:13} is $\sum_{C\in\mathcal{C}^{(l-1)}\cap C_{0}}\sum_{\mathfrak{i}\in\mathcal{I}^{(l-1)}\cap C}\vert\mathcal{V}_{\delta,C,\mathfrak{i}}\vert = \sum_{C\in\mathcal{C}^{(l-1)}\cap C_{0}}\vert\mathcal{V}_{\delta,C}\vert = 16r*2^{l-1}$. Therefore, computing the basis functions of $\mathcal{F}_{\delta}^{(l)}$ localized within cell $C_{0}$  costs $O(r^{3}2^{3l})$. Since there are $\vert \mathcal{C}^{(l)}\vert=\frac{I^{2}}{2^{2l}}$ cells on level $l$, the total cost on level $l$ is $O(r^{3}2^{3l})*O(\frac{I^{2}}{2^{2l}})=O(2^{l}r^{3}I^{2})$, and the same estimate holds for $\mathcal{F}_{\delta}^{(0)}$ and $\mathcal{G}_{\delta}^{(l)}$ for $1 \le l \le L$. Thus, the overall cost for computing all explicit basis functions across all levels is $\sum_{l=0}^{L}O(2^{l}r^{3}I^{2})=O(r^{3}I^{3})$.

In practice, the cell-wise computations at each level are mutually independent and can be fully parallelized. Under parallel execution, the work per level reduces to the cost of a single cell, $O(r^3 2^{3l})$, yielding a parallel time complexity of $\sum_{l=0}^{L} O(r^3 2^{3l}) = O(r^3 2^{3L}) = O(r^3 I^3)$ 
Although the asymptotic complexity remains unchanged, the inherent parallelism across independent cells substantially reduces the execution time, leading to significant acceleration in the offline assembly phase.

\end{proof}

\begin{lemma}\label{lemma:online_computation_cost}
    The online computational cost for a single time step, corresponding to applying the multilevel solution operator within the fixed-point iteration \eqref{eq:sec2:21}, scales as $O(\Delta t^{-1} r^2 I^2 \log I)$.
\end{lemma}
\begin{proof}
    By construction, $B_{\delta}^{(l)}$ is a block-diagonal matrix comprising $|\mathcal{C}^{(l)}| = I^2 2^{-2l}$ blocks, each of size $O(r 2^l) \times O(r 2^l)$. Consequently, applying $B_{\delta}^{(l)}$ to a vector requires $O(r^2 I^2)$ floating-point operations. The same cost bound holds for $Q_{\delta}^{(l)}$ and $P_{\delta}^{(l)}$. For the restriction operators $R_{\delta}^{(l)}$ and $\check{R}_{\delta}^{(l)}$, the per-level cost is $O(r I^2 2^{-l})$. Summing over all levels $l = 0, \dots, L$, the total cost of the solution operator $B_{\delta}^{-1}$ acting on an arbitrary right hand side vector is:
    \[
    \sum_{l=0}^{L} \left[ O(r^2 I^2) + O(r I^2 2^{-l}) \right] = O(L r^2 I^2) = O(r^2 I^2 \log I),
    \]
    where we have used $L \sim \log_2 I$. The fixed-point iteration \eqref{eq:sec2:21} exhibits a contraction factor proportional to $(1 - \Delta t)$; hence, the number of inner iterations required to reach a prescribed tolerance $\mathrm{tol}$ scales as $O(\Delta t^{-1})$. Multiplying the per-iteration cost by the iteration count gives the total sequential complexity $O(\Delta t^{-1} r^2 I^2 \log I)$.

    With full parallelization, the computations involving sub-blocks of $B_{\delta}^{(l)}$ acting on a vector, as well as the computations across sub-block rows, can be performed in parallel. Therefore, applying $B_{\delta}^{(l)}$ to a vector in a fully parallel fashion requires at most $O(r2^{l})$ time. The same parallelization efficiency is achievable for $Q_{\delta}^{(l)}$, $P_{\delta}^{(l)}$, $R_{\delta}^{(l)}$ and $\check{R}_{\delta}^{(l)}$. Thus, the total parallel cost of $B_{\delta}^{-1}$ acting on an arbitrary right hand side vector should be:
    \[
     \sum_{l=0}^{L} O(r 2^{l})   = O(r2^{L}) = O(rI),
    \]
    Considering the $O(\Delta t^{-1})$ inner iterations, the overall parallel computational time is reduced to $O(\Delta t^{-1} rI)$, resulting in a significant acceleration during the online phase.

\end{proof}

\begin{remark}
    Because both the time and space discretizations are second-order accurate and the time stepping is implicit, we do not actually need very small values of $\Delta t$ or the cell width $h$ to get an accurate solution. In that case, the computational cost per time step, $O(\Delta t^{-1} r^2 I^2)$, is acceptable. On the other hand, if a very small $\Delta t$ is truly necessary, we can instead use the iteration in \eqref{eq:sec2:24}, where the number of iterations is bounded above and does not scale as $\Delta t^{-1}$.
\end{remark}

\begin{lemma}\label{lemma:storage_cost}
The storage cost for the multilevel decomposition of the solution operator $B_{\delta}^{-1}$ is $O(r^{2} I^{2} \log(I))$.
\end{lemma}
\begin{proof}
According to Lemma \ref{lemma:twoleveldecomposition}, the multilevel decomposition of $B_{\delta}^{-1}$ consists of the operators  
$B_{\delta}^{(l)}$ for $0 \leq l \leq L$,  
$P_{\delta}^{(l)}$, $Q_{\delta}^{(l)}$, $R_{\delta}^{(l)}$, and $\check{R}_{\delta}^{(l)}$ for $0 \leq l \leq L-1$,  
and the coarsest level inverse $B_{\delta}^{(L),-1}$.  
We now analyze their storage costs level by level.

By definition, for any $f \in \mathcal{F}_{\delta}^{(l)}$, the operator $B_{\delta}^{(l)}$ maps $\mathbf{f}_{\mathcal{C}^{(l)}}$ to 
$(\mathbf{f}_{\mathrm{b}}^{T}, [\mathbf{f}]_{\mathcal{I}^{(l)}_{\mathrm{in}}}^{T})^{T}$. Since $\mathbf{f}_{\mathrm{b}}$ is a subset of $\mathbf{f}_{\mathcal{C}^{(l)}}$, storing this correspondence requires only $O(|\mathcal{F}_{\delta}^{(l)}|) = O(\frac{rI^{2}}{2^{l}})$ memory. Moreover, by Lemma \ref{lemma:expression}, the mapping to $[\mathbf{f}]_{\mathcal{I}^{(l)}_{\mathrm{in}}}$ is fully determined by the jump values $[\bm{\phi}_{C,\mathfrak{i}}^{(l),k}]_{\mathcal{I}^{(l)} \cap C}$ of each basis function $\phi_{C,\mathfrak{i}}^{(l),k} \in \mathcal{F}_{\delta}^{(l)}$. Storing these values incurs a cost of $O(r 2^{\,l}) \cdot O\!\left(\frac{r I^2}{2^{\,l}}\right) = O(r^2 I^2)$ memory, independent of the mesh level $l$.

Similarly, $P_{\delta}^{(l)}$ and $Q_{\delta}^{(l)}$ store $(\bm{\phi}^{(l+1),k}_{C,\mathfrak{i}})_{\mathcal{I}^{(l)} \cap C}$ for each basis function $\phi_{C,\mathfrak{i}}^{(l+1),k}$ in $ \mathcal{F}_{\delta}^{(l+1)}$ and $ \mathcal{G}_{\delta}^{(l+1)}$, respectively, which also requires $O(r^{2} I^{2})$ storage per level.

Furthermore, $R_{\delta}^{(l)}$ and $\check{R}_{\delta}^{(l)}$ are binary (0–1) matrices encoding the nested structure of the solution spaces, and their storage cost is $O(|\mathcal{F}_{\delta}^{(l)}|) = O\!\left(\frac{r I^{2}}{2^{\,l}}\right)$.

Finally, since $B_{\delta}^{(L)}\in\mathbb{R}^{\vert\mathcal{F}_{\delta}^{(L)}\vert\times\vert\mathcal{F}_{\delta}^{(L)}\vert}$ and $|\mathcal{F}_{\delta}^{(L)}| = O(rI)$, the dense inverse $B_{\delta}^{(L),-1}$ requires $O(|\mathcal{F}_{\delta}^{(L)}|^{2}) = O(r^{2}I^{2})$ storage.

Therefore, the overall storage cost becomes $\sum_{l=0}^{L}O(r^{2} I^{2}) = O(r^{2}I^{2}\log(I))$. 

\end{proof}

\section{Numerical Experiments}
\label{sec:experiments}
This section presents numerical experiments to validate the accuracy and the asymptotic preserving (AP) property of the proposed discretization scheme. Furthermore, using benchmark multiscale problems, we verify the scheme's capability to adaptively compress the angular domain in the presence of local low-rank structures.

In the following, we introduce and compare three types of solutions:
\begin{itemize}
    \item \textbf{The exact solution}, which is the exact solution to the discrete ordinate transport equations \eqref{eq:sec3:23}, \eqref{eq:sec2:27}, \eqref{eq:sec2:28}, and \eqref{eq:sec2:29}.
    \item \textbf{The full-order solution}, which is the numerical solution to \eqref{eq:sec3:23}, \eqref{eq:sec2:27}, \eqref{eq:sec2:28}, and \eqref{eq:sec2:29} using TFPS for spatial discretization, denoted by $\tilde{\bm{\psi}}$.
    \item \textbf{The low-rank solution}, which is the numerical solution to \eqref{eq:sec3:23}, \eqref{eq:sec2:27}, \eqref{eq:sec2:28}, and \eqref{eq:sec2:29} using ATFPS for spatial discretization with layer reconstruction, denoted by $\tilde{\bm{\psi}}_{\delta}^{*}$.
\end{itemize}

To quantify the reduction in degrees of freedom achieved by the low rank scheme, we define the rank ratio as:
\begin{equation}\label{eq:sec5:3}
    \mathrm{rank\ ratio}=\frac{\sum_{C\in\mathcal{C}}\vert\mathcal{V}_{\delta,C}\vert}{\sum_{C\in\mathcal{C}}\vert\mathcal{V}\vert},
\end{equation}
which represents the ratio of the total number of ATFPS bases to the total number of TFPS bases in the spatial discretization. 

To assess the accuracy of the low-rank solution against the full-order or exact solution, we introduce the discrete error norm $\Vert \cdot\Vert_{I\times I}$. For any $4M$-dimensional vector-valued function $f(x,y)$ on a spatial domain $\Omega=[a,b]\times[a,b]$, uniformly discretized by an $I\times I$ grid (where $h=\frac{b-a}{I}$, $x_{i}=ih$, and $y_{j}=jh$), the norm is defined as:
\begin{equation*}
    \Vert f\Vert_{I\times I}=\frac{h}{\sqrt{M}}\Big(\sum_{i,j=1}^{I}\sum_{m\in\mathcal{M}}(f_{m}(x_{i-1/2},y_{j-1/2}))^{2}\Big)^{1/2}.
\end{equation*}
Note that this norm extends to scalar functions by setting $f_{m}=f$ for all $m\in\mathcal{M}$ to construct the vector $(f_{1},f_{2},\dots,f_{4M})^{T}$. Then The relative $L_{2}$ error between $\tilde{\bm{\psi}}$ and $\tilde{\bm{\psi}}_{\delta}^{*}$ on the $I\times I$ mesh is given by $\frac{\Vert\tilde{\bm{\psi}}- \tilde{\bm{\psi}}_{\delta}^{*}\Vert_{I\times I}}{\Vert\tilde{\bm{\psi}}\Vert_{I\times I}}$. In the following, we take $I=1024$.

\begin{remark}
A numerical scheme is termed \textit{asymptotic preserving} (AP) if it correctly captures the limiting macroscopic equation (in this case, the diffusion limit) at the discrete level, without requiring the spatial mesh or time step to resolve microscopic scales. This property guarantees numerical robustness in the asymptotic regime as certain parameter becomes arbitrarily small. 

It has been established that the TFPS possesses the AP property for the steady-state RTE \cite{han2014two,wang2022uniform}. Consequently, the adaptive variant (ATFPS) inherits this AP character \cite{AdaptiveTFPS}. However, the AP property of a spatial discretization does not automatically extend to time-dependent problems \cite{hu2017asymptotic,jin2022asymptotic}. For instance, \cite{jin2022asymptotic} demonstrates that coupling a Crank--Nicolson time integrator with a pseudo-spectral spatial discretization may fail to recover the correct diffusion limit, even for vanishingly small time steps (e.g., $\Delta t = 10^{-4}$). In contrast, employing Strang splitting for the temporal discretization alongside the same spatial scheme successfully preserves the asymptotic behavior. Therefore, a rigorous verification of the AP property for the fully discrete scheme is essential.

\end{remark}
\subsection{Accuracy and efficiency test}
\label{subsec:accuracy_test}
We begin by assessing the accuracy and efficiency of our discretization scheme using a manufactured solution with constant coefficients. Let the exact solution be defined as: 
\[
\bm{\psi}_{\mathrm{exact}}(x,y,\mathbf{u},t)=\frac{t}{1+t}(1+\epsilon x)\xi_{\mathrm{min}}\exp(\frac{\sqrt{3}}{2}\lambda_{\mathrm{min}}x+\frac{1}{2}\lambda_{\mathrm{min}}y),
\]
where $\lambda_{\mathrm{min}}$ is the negative eigenvalue of smallest magnitude for the matrix given below, and $\xi_{\mathrm{min}}$ is the corresponding eigenvector. The matrix is defined as:
\[
(\frac{\sqrt{3}}{2}C+\frac{1}{2}S)^{-1}\Big((\frac{\sigma_{T}}{\epsilon}-\epsilon\sigma_{a})W-\frac{\sigma_{T}}{\epsilon}I\Big)
\]
with $\sigma_{T}$, $\epsilon$ and $\sigma_{a}$ being constant and 
\[
C=\mathrm{diag}\{c_{1}, c_{2},\dots,c_{4M}\},\quad S=\mathrm{diag}\{s_{1}, s_{2}, \dots, s_{4M},\quad W = 
\begin{pmatrix}
    \omega_{1} & \omega_{2} & \dots &\omega_{4M}\\
    \omega_{1} & \omega_{2} & \dots &\omega_{4M}\\
    \vdots &  \vdots & \vdots & \vdots \\
    \omega_{1} & \omega_{2} & \dots &\omega_{4M}\\
\end{pmatrix}.
\]
Consequently, $\bm{\psi}_{\mathrm{exact}}$ serves as the exact solution for the isotropic ($\kappa\equiv 1$) time-dependent RTE \eqref{eq:sec1:1} with angular discretization, provided that the external source term is given by:
\[
q_{\mathrm{exact}}=\left[\frac{1}{(t+1)^{2}}(1+\epsilon x)\xi_{\mathrm{min}}+\frac{t}{1+t}C\xi_{\mathrm{min}}\right]\exp\left(\frac{\sqrt{3}}{2}\lambda_{\mathrm{min}}x+\frac{1}{2}\lambda_{\mathrm{min}}y\right).
\]
Additionally, the initial condition is set to $\bm{\psi}(x,y,\mathbf{u},0)=0$, and the inflow boundary condition $\bm{\psi}(x,y,\mathbf{u},t)|_{\Gamma^{-}}$ is set to match $\bm{\psi}_{\mathrm{exact}}|_{\Gamma^{-}}$.

The numerical experiments are configured with the following parameters: $\sigma_T = 1$, $\sigma_a = 0.5$, final time $T = 1$, and basis selection threshold $\delta = 10^{-3}$. The angular domain is discretized using 4, 12, and 24 discrete ordinates, corresponding to $M = 1$, $3$, and $6$, respectively. To evaluate the scheme's performance across regimes ranging from kinetic to diffusive, we vary the scaling parameter (mean free path) $\varepsilon \in \{1/2, 1/8, 1/32, 1/128, 1/512\}$ and the spatial mesh size $h \in \{1/4, 1/8, 1/16, 1/32, 1/64\}$, while fixing the time step at $\Delta t = h$. 
For each parameter combination, we report the relative $L^2$ error between the numerical and exact solutions at $T = 1$ to quantify accuracy. Computational efficiency is assessed by recording the execution time for the offline phase (assembly of the multilevel decomposition of the solution operator) and the online phase (application of the operator to the right-hand side vectors at each iteration). To accelerate computations, GPU parallelization is leveraged for all operations within each mesh level during both phases. The algorithm is implemented in Python using the PyTorch framework, and all numerical experiments are conducted on a workstation equipped with an NVIDIA RTX A6000 GPU.

Figure \ref{fig:accuracy_test} shows how the relative $L_{2}$ error changes with $h$ for each fixed $M$ and $\epsilon$, which demonstrate that the proposed discretization scheme achieves uniform second-order accuracy with respect to $\epsilon$, thereby confirming its asymptotic preserving (AP) property. Notably, our method maintains its accuracy even in the intermediate regime where $\epsilon$ and $h$ are of comparable size. This robustness highlights a distinct advantage of the proposed scheme.

    

\begin{figure}
    \centering
    \includegraphics[width=0.9\linewidth]{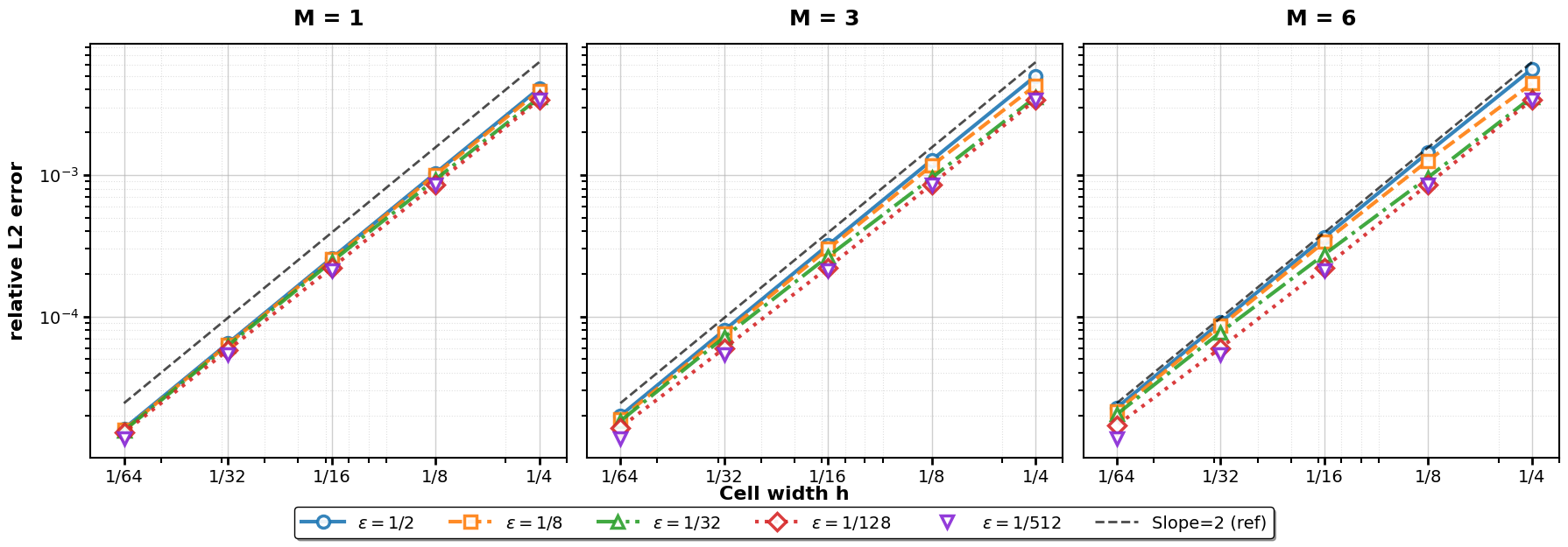}
    \caption{Relative $L_2$ error between exact and numerical solutions versus cell width $h$ for different numbers of velocity directions. }
    \label{fig:accuracy_test}
\end{figure}

Figures \ref{fig:efficiency_test_offline} and Table \ref{tab:efficiency_test_offline} summarize the computational efficiency of the offline assembly phase. Figure \ref{fig:efficiency_test_offline} plots the computation time required to construct the multilevel solution operator against the spatial mesh size $h$ for various $\epsilon$ and $M$, while Table \ref{tab:efficiency_test_offline} reports the estimated order of offline assembling time with respect to the cell width $h$ based on consecutive $h$ levels for each fixed $M$ and $\epsilon$. The results demonstrate that, by leveraging GPU parallelization, the offline assembly time gradually approaches $O(h^{-3})$, which is well within the theoretical upper bound $O(h^{-3})$ derived in the complexity analysis.

\begin{figure}
    \centering
    \includegraphics[width=0.9\linewidth]{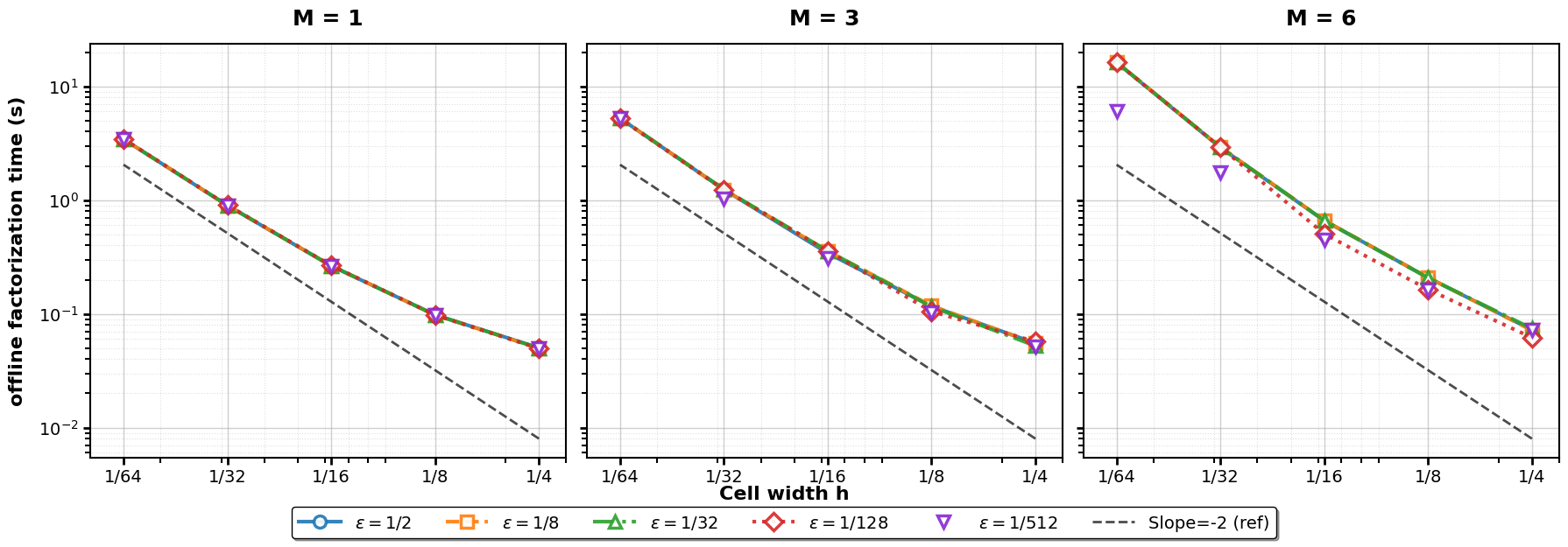}
    \caption{Offline assembling time for the multilevel decomposition of the solution operator versus cell width $h$, with varying velocity directions. }
    \label{fig:efficiency_test_offline}
\end{figure}

\begin{table}[htbp]
\centering
\caption{Estimated convergence order of offline assembling time with respect to $h$ based on consecutive $h$ refinements for varying $\varepsilon$ and $N$}
\label{tab:efficiency_test_offline}
\begin{tabular}{c c|cccc}
\toprule
\multicolumn{2}{c|}{} & \multicolumn{4}{c}{$h$ pairs} \\
\cmidrule(lr){3-6}
$M$ & $\varepsilon$ & $1/4 \to 1/8$ & $1/8 \to 1/16$ & $1/16 \to 1/32$ & $1/32 \to 1/64$ \\
\midrule
\multirow{5}{*}{1} & $1/2$   & 0.96 & 1.43 & 1.77 & 1.94 \\
                    & $1/8$   0.97 & 1.43 & 1.76 & 1.93 \\
                    & $1/32$  & 0.97 & 1.44 & 1.76 & 1.93 \\
                    & $1/128$ & 0.98 & 1.44 & 1.76 & 1.93 \\
                    & $1/512$ & 0.96 & 1.43 & 1.77 & 1.95 \\
\midrule
\multirow{5}{*}{3} & $1/2$   & 1.08 & 1.55 & 1.86 & 2.08 \\
                    & $1/8$   & 1.10 & 1.59 & 1.80 & 2.08 \\
                    & $1/32$  & 1.14 & 1.63 & 1.80 & 2.08 \\
                    & $1/128$ & 0.88 & 1.76 & 1.80 & 2.08 \\
                    & $1/512$ & 1.01 & 1.57 & 1.74 & 2.34 \\
\midrule
\multirow{5}{*}{6} & $1/2$   & 1.53 & 1.66 & 2.14 & 2.48 \\
                    & $1/8$   & 1.54 & 1.67 & 2.14 & 2.48 \\
                    & $1/32$  & 1.48 & 1.67 & 2.14 & 2.48 \\
                    & $1/128$ & 1.41 & 1.62 & 2.53 & 2.47 \\
                    & $1/512$ & 1.18 & 1.47 & 1.97 & 1.78 \\
\bottomrule
\end{tabular}
\end{table}

Figures \ref{fig:efficiency_test_online} and Table \ref{tab:efficiency_test_online} present the performance of the online solution stage. Figure \ref{fig:efficiency_test_online} depicts the total online computation time for advancing all time steps as a function of $h$, and Table \ref{tab:efficiency_test_online} presents the estimated order of the total online computation time  with respect to the cell width $h$ based on consecutive $h$ levels for each fixed $M$ and $\epsilon$. The data reveal that the online solution time scales approximately as $O(h^{-2.5})$. According to Lemma \ref{lemma:online_computation_cost}, the theoretical parallel complexity for a single time step is $O(\Delta t^{-1} r I)$. With $\Delta t = h$ and $I \propto h^{-1}$, the worst-case total time over all time steps scales as $O(\Delta t^{-2} r I) = O(h^{-3})$. The empirically observed growth of approximately $O(h^{-2.5})$ is milder than this theoretical worst-case bound, indicating that the actual computational workload may benefit either from favorable spectral properties in practice or from better GPU parallelism during runtime.
\begin{figure}
    \centering
    \includegraphics[width=0.9\linewidth]{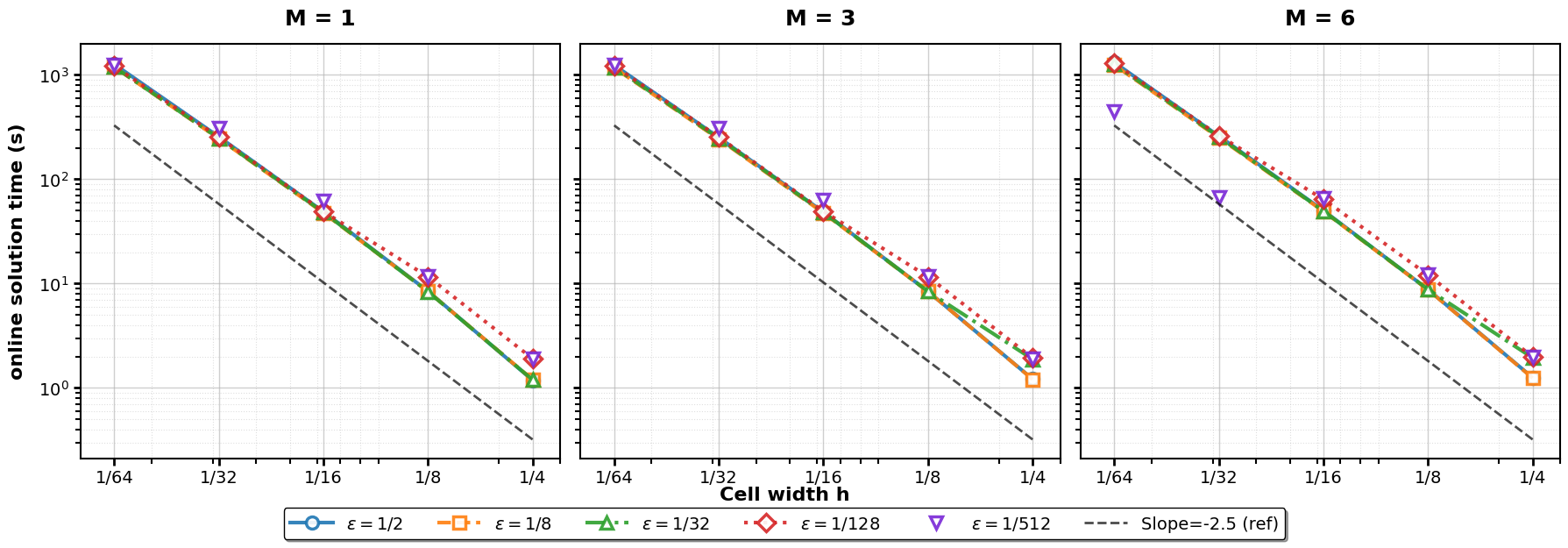}
    \caption{Online solution time for applying the solution operator versus cell width $h$ for different numbers of velocity directions.}
    \label{fig:efficiency_test_online}
\end{figure}

\begin{table}[htbp]
\centering
\caption{Estimated convergence order of online solution time with respect to $h$ based on consecutive $h$ refinements for varying $\varepsilon$ and $N$}
\label{tab:efficiency_test_online}
\begin{tabular}{c c|cccc}
\toprule
\multicolumn{2}{c|}{} & \multicolumn{4}{c}{$h$ pairs} \\
\cmidrule(lr){3-6}
$M$ & $\varepsilon$ & $1/4 \to 1/8$ & $1/8 \to 1/16$ & $1/16 \to 1/32$ & $1/32 \to 1/64$ \\
\midrule
\multirow{5}{*}{1} & $1/2$   & 2.83 & 2.53 & 2.38 & 2.31 \\
                    & $1/8$   & 2.81 & 2.50 & 2.36 & 2.28 \\
                    & $1/32$  & 2.81 & 2.53 & 2.36 & 2.27 \\
                    & $1/128$ & 2.61 & 2.08 & 2.37 & 2.27 \\
                    & $1/512$ & 2.63 & 2.41 & 2.30 & 2.01 \\
\midrule
\multirow{5}{*}{3} & $1/2$   & 2.79 & 2.53 & 2.37 & 2.30 \\
                    & $1/8$   & 2.79 & 2.50 & 2.35 & 2.27 \\
                    & $1/32$  & 2.16 & 2.52 & 2.36 & 2.26 \\
                    & $1/128$ & 2.59 & 2.08 & 2.36 & 2.26 \\
                    & $1/512$ & 2.62 & 2.41 & 2.30 & 2.01 \\
\midrule
\multirow{5}{*}{6} & $1/2$   & 2.81 & 2.52 & 2.37 & 2.35 \\
                    & $1/8$   & 2.81 & 2.50 & 2.34 & 2.31 \\
                    & $1/32$  & 2.18 & 2.50 & 2.36 & 2.31 \\
                    & $1/128$ & 2.61 & 2.41 & 2.03 & 2.31 \\
                    & $1/512$ & 2.63 & 2.43 & 0.04 & 2.72 \\
\bottomrule
\end{tabular}
\end{table}

\subsection{Benchmark multiscale tests}
\label{subsec:benchmark}

In this subsection, we evaluate the scheme's performance on two multiscale benchmark problems: the lattice case and the bufferzone case \cite{AdaptiveTFPS}. The study aims to demonstrate the effectiveness of our discretization scheme in capturing local low-rank structures in the angular domain, and to examine the trade-off between accuracy and efficiency using the adjustable threshold $\delta$ (or equivalently, the rank ratio).

In the following case, we consider a system with zero initial condition and zero external source. The time-dependent isotropic inflow boundary condition is defined as:
\[
\psi_{\Gamma^{-}}=\frac{t}{1+t}.
\]
The computational parameters include a final simulation time of $T = 1$, a time-step size of $\Delta t = 1/32$, resulting in $N = T / \Delta t = 32$ time steps, a spatial grid of $I \times I = 32 \times 32$, and angular directions $4M \in \{24, 40, 60\}$.

\subsubsection{Lattice case}
In the lattice example, two materials with different optical properties coexist, creating sharp interfaces between them, as illustrated in Figure~\ref{fig:benchmark_lattice_test}(A). We set $\sigma_{T}=1$ and $\sigma_{a}=0.5$ in the computational domain $\Omega = [0,1]\times[0,1]$. In the blue rectangles (diffusion regime), we choose $\epsilon = 0.01$, whereas in the yellow rectangles (transport regime), we set $\epsilon = 1$.

Figure \ref{fig:benchmark_lattice_test} presents the numerical results for the lattice benchmark case. The figure contains five panels besides lattice layout illustrating the scheme’s efficiency and accuracy: (B) Basis Function Count: The spatial distribution of the number of basis functions in each cell for the low-rank solution (at a rank ratio of $60\%$); (C) Full-Order Scalar Flux: The full-order scalar flux $\tilde{\phi}^{N}$ at $T=1$ for a chosen number of velocity directions ($M=10$). (D) Low-Rank Scalar Flux: The corresponding low-rank scalar flux $\tilde{\phi}_{\delta}^{*,N}$ for $M=10$ at $T=1$ (at a rank ratio of $60\%$). (E) and (F) Error v.s. rank ratio: The evolution of the relative $L_{2}$ errors versus the rank ratio for different numbers of velocity directions ($M = 3, 6, 10$, or $15$, corresponding to $12, 24, 40$, and $60$ discrete velocity directions, respectively). Both the error in the angular flux ($\tilde{\bm{\psi}}^{N}$ versus its low-rank approximation $\tilde{\bm{\psi}}_{\delta}^{*,N}$) and the error in the scalar flux ($\tilde{\phi}^{N}$ versus $\tilde{\phi}_{\delta}^{*,N}$) are tracked.

\begin{figure}[htbp]
    \centering
    \begin{minipage}{0.45\textwidth}
        \centering
        \includegraphics[width=\linewidth]{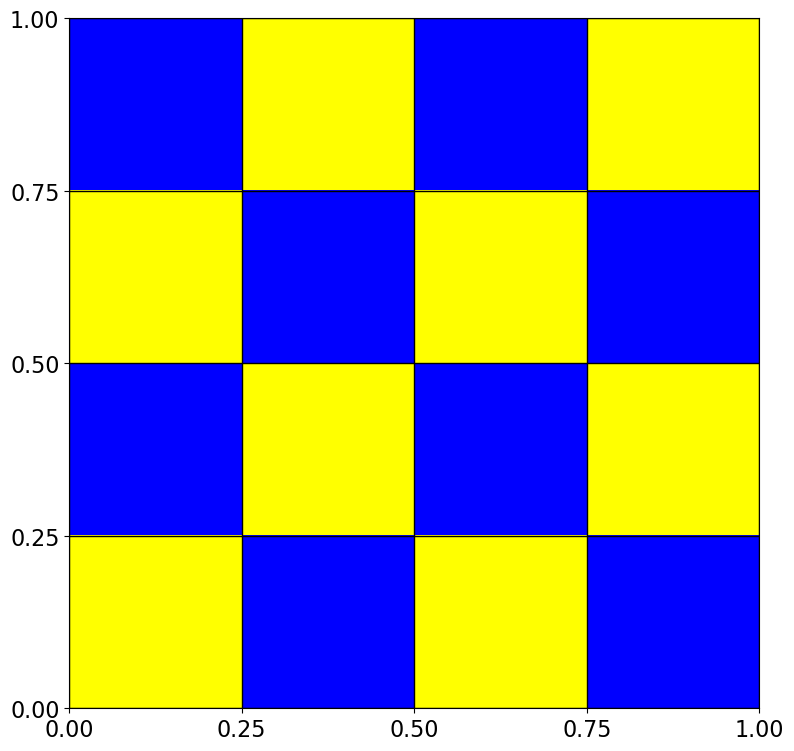}
    \small (a)
    \end{minipage}
    \begin{minipage}{0.45\textwidth}
        \centering
        \includegraphics[width=\linewidth]{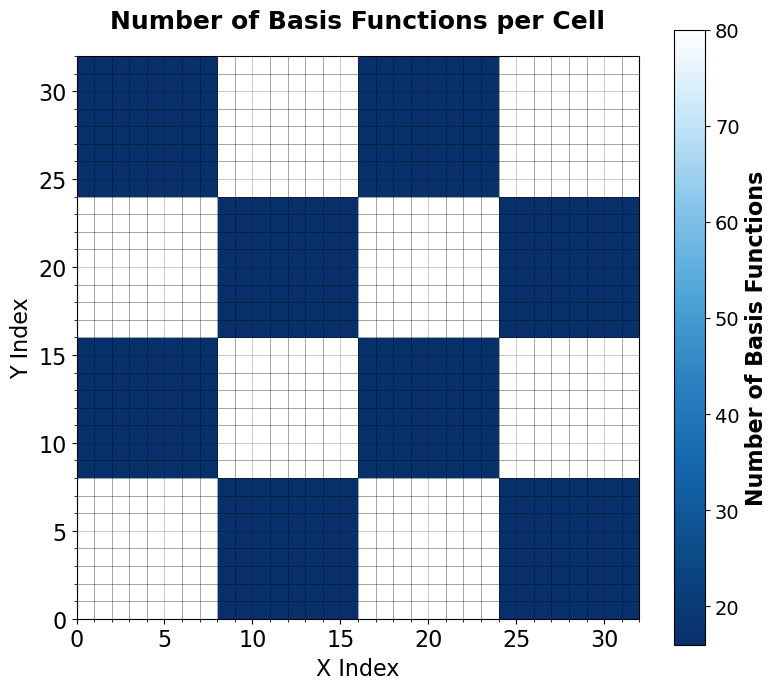}
    \small (b)
    \end{minipage}

    \begin{minipage}{0.45\textwidth}
        \centering
        \includegraphics[width=\linewidth]{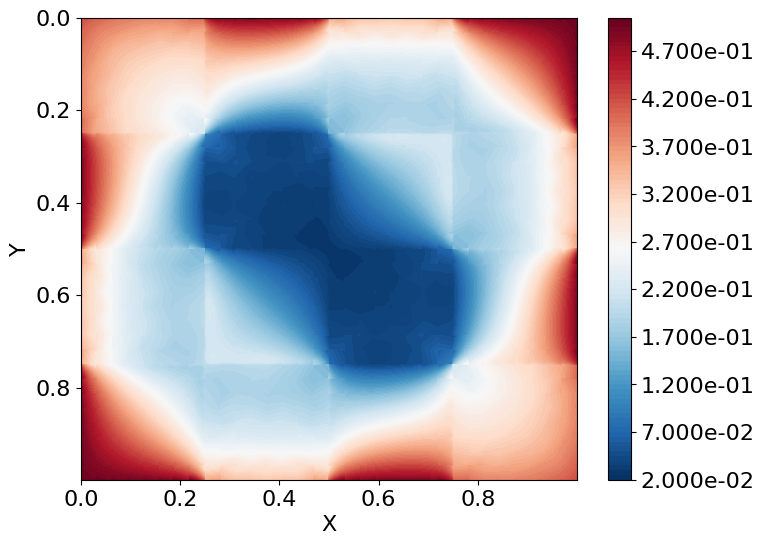}
    \small (c)
    \end{minipage}
    \begin{minipage}{0.45\textwidth}
        \centering
        \includegraphics[width=\linewidth]{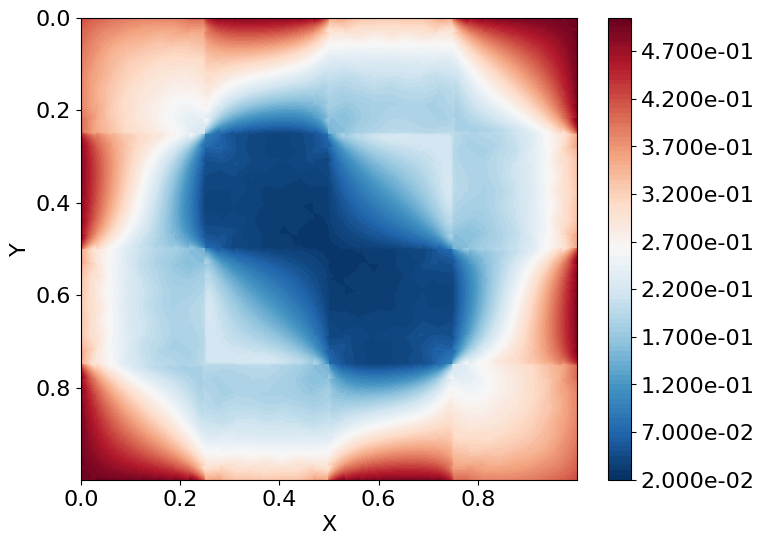}
    \small (d)
    \end{minipage}

    \begin{minipage}{0.45\textwidth}
        \centering
        \includegraphics[width=\linewidth]{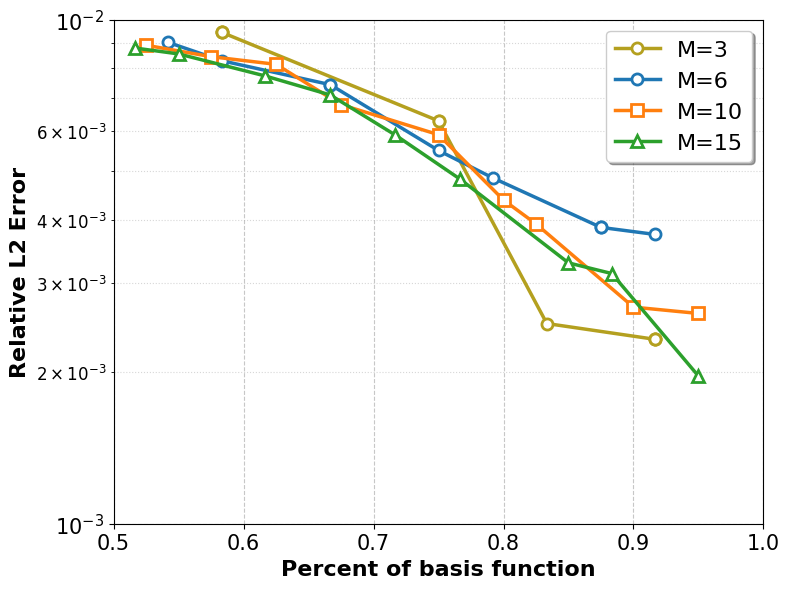}
    \small (e)
    \end{minipage}
    \begin{minipage}{0.45\textwidth}
        \centering
        \includegraphics[width=\linewidth]{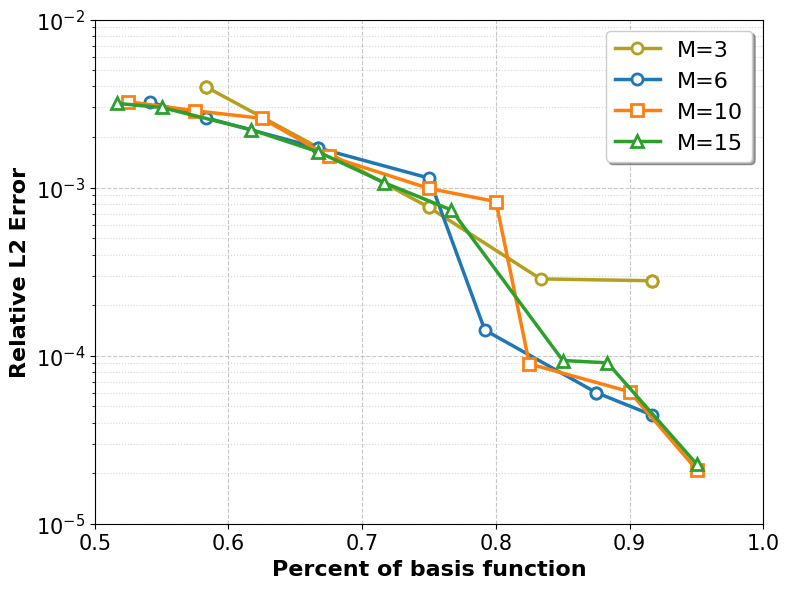}
    \small (f)
    \end{minipage}
    \caption{Numerical results in lattice case. (a) Layout for lattice case; (b) The number of basis functions in each physical cell for the low rank solution when rank ratio is $60\%$; (c) The full order solution at $t=1$s; (d) The low rank solution at $t=1$s; (e) error v.s. rank ratio for angular flux; (f) error v.s. rank ratio for scalar flux.}
    \label{fig:benchmark_lattice_test}
\end{figure}

The results demonstrate that our discretization scheme effectively compresses the angular domain in the presence of local low-rank structures. As shown in Figure~\ref{fig:benchmark_lattice_test}(B), the number of basis functions remains uncompressed in the transport regime. In contrast, in the diffusion regime, the number of basis functions per cell is significantly reduced to four, which is consistent with the asymptotic analysis. Moreover, the comparison between Figures~\ref{fig:benchmark_lattice_test}(C) and~\ref{fig:benchmark_lattice_test}(D) confirms the high accuracy of our compression approach. Finally, Figures~\ref{fig:benchmark_lattice_test}(E) and~(F) demonstrate that adjusting the threshold $\delta$ allows for a controlled trade-off between accuracy and computational efficiency.

\subsubsection{bufferzone case}
Next, we consider the bufferzone case where the total cross section $\sigma_{T}$, absorption corss section $\sigma_{a}$ and mean free path $\epsilon$ are selected as follows
\[
\Sigma_{T} = 1+x^2+y^2, \ \sigma_a=0.5+x^2+y^2, \ \epsilon=0.02x+0.001, \quad (x,y)\in\Omega=[0,1]\times[0,1].
\]

Figure~\ref{fig:benchmark_bufferzone_test} illustrates the numerical results for the buffer-zone scenario. The figure comprises six panels: (A) and (B) Basis Function Count: Spatial distributions of the number of basis functions per cell in the low-rank solution, shown for rank ratios of $20\%$ and $40\%$, respectively when $M=10$. (C) Full-Order Scalar Flux: The full-order scalar flux $\tilde{\phi}^{N}$ at $T = 1$ for a selected number of velocity directions ($M = 10$). (D) Low-Rank Scalar Flux: The corresponding low-rank scalar flux $\tilde{\phi}_{\delta}^{*,N}$ at $T = 1$ for $M = 10$, with a rank ratio of $20\%$. (E) and (F) Error v.s. Rank Ratio: The variation of the relative $L_2$ errors as a function of the rank ratio for different numbers of velocity directions($M = 3, 6, 10$, or $15$, corresponding to $12, 24, 40$, and $60$ discrete velocity directions, respectively), including both the error in the angular flux ($\tilde{\bm{\psi}}^{N}$ versus its low-rank approximation $\tilde{\bm{\psi}}_{\delta}^{*,N}$) and the error in the scalar flux ($\tilde{\phi}^{N}$ versus $\tilde{\phi}_{\delta}^{*,N}$).

\begin{figure}[htbp]
    \centering
    \begin{minipage}{0.45\textwidth}
        \centering
        \includegraphics[width=\linewidth]{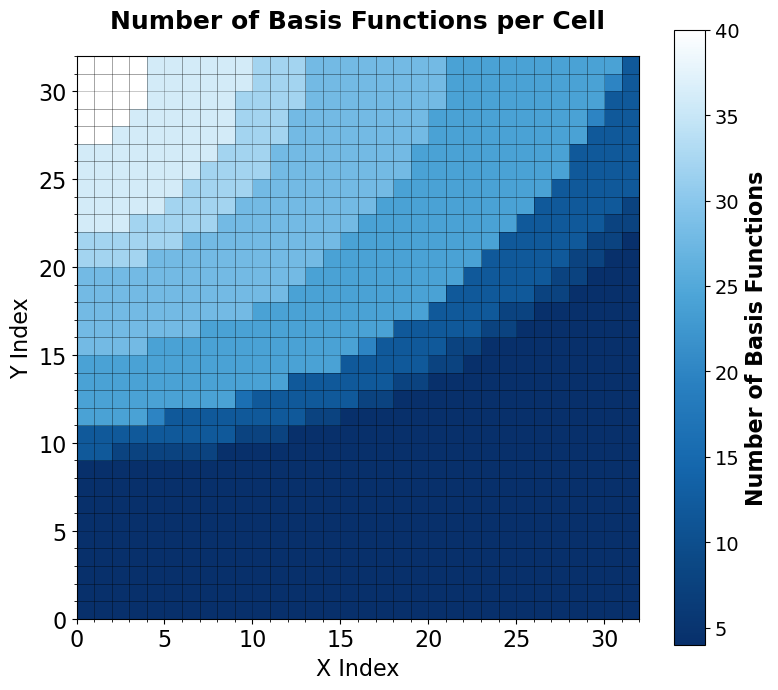}
    \small (a)
    \end{minipage}
    \begin{minipage}{0.45\textwidth}
        \centering
        \includegraphics[width=\linewidth]{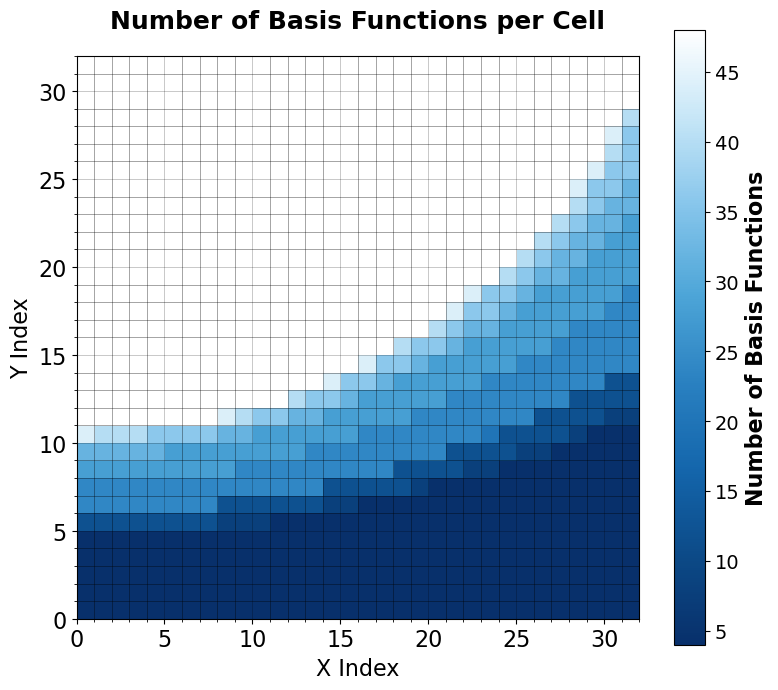}
    \small (b)
    \end{minipage}

    \begin{minipage}{0.45\textwidth}
        \centering
        \includegraphics[width=\linewidth]{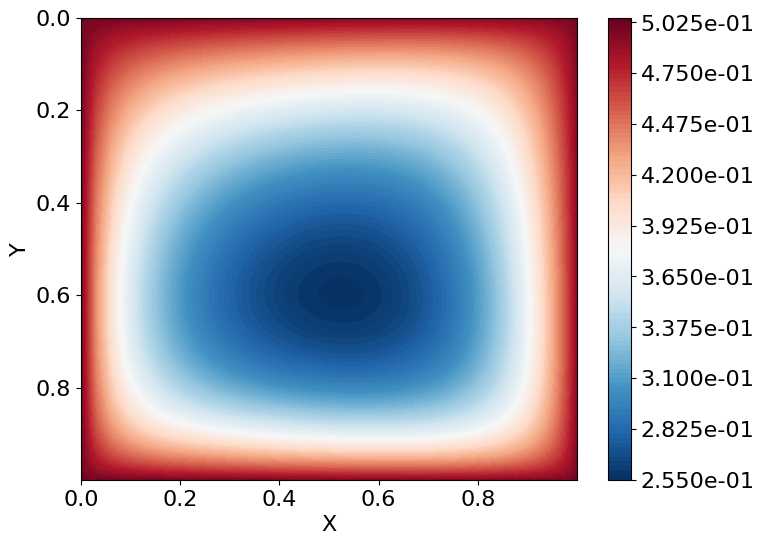}
    \small (c)
    \end{minipage}
    \begin{minipage}{0.45\textwidth}
        \centering
        \includegraphics[width=\linewidth]{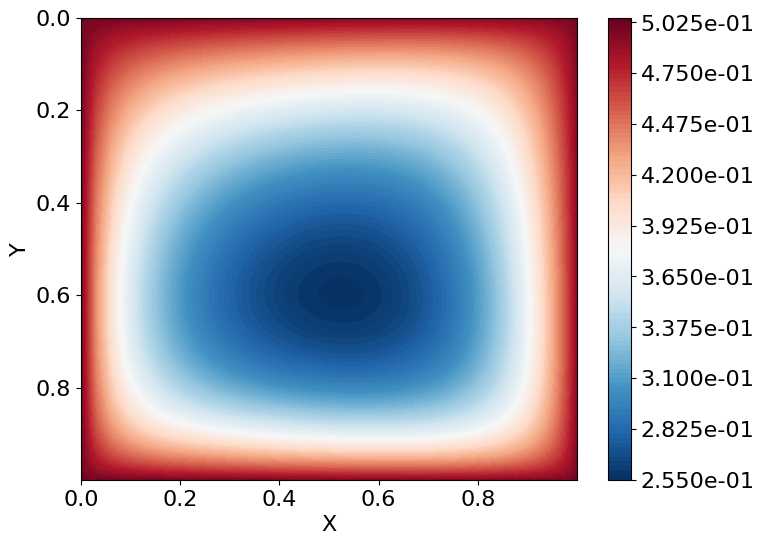}
    \small (d)
    \end{minipage}

    \begin{minipage}{0.45\textwidth}
        \centering
        \includegraphics[width=\linewidth]{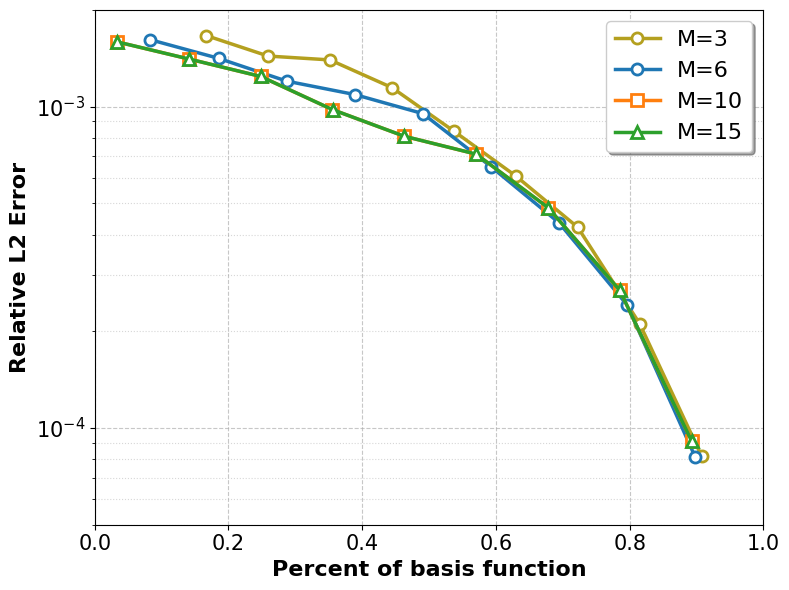}
    \small (e)
    \end{minipage}
    \begin{minipage}{0.45\textwidth}
        \centering
        \includegraphics[width=\linewidth]{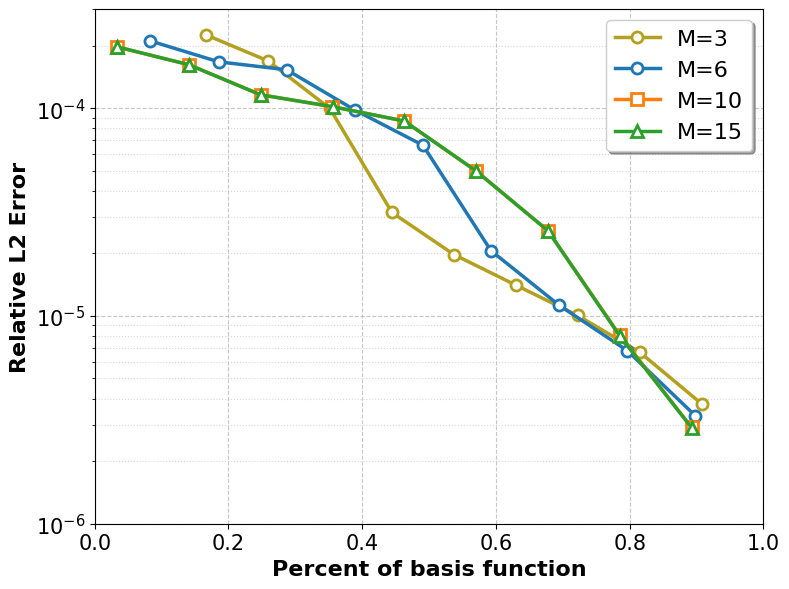}
    \small (f)
    \end{minipage}
    \caption{Numerical results in bufferzone case. (a) The number of basis functions in each physical cell for the low rank solution when rank ratio is $40\%$; (b) The number of basis functions in each physical cell for the low rank solution when rank ratio is $20\%$; (c) The full order solution at $t=1$s; (d) The low rank solution at $t=1$s when rank ratio is $40\%$; (e) error v.s. rank ratio for angular flux; (f) error v.s. rank ratio for scalar flux.}
    \label{fig:benchmark_bufferzone_test}
\end{figure}

The results demonstrate the adaptive compression effect of the ATFPS method in the presence of a buffer zone. As shown in Figures~\ref{fig:benchmark_bufferzone_test}(B) and (C), when the material gradually transitions from the transport regime to the diffusion regime, the number of basis functions at the corresponding locations decreases progressively from $8M$ (the fully uncompressed number) to 4 (the fully compressed number). Moreover, increasing the rank ratio results in a continuous increase of the number of basis functions per cell.  
Figures~\ref{fig:benchmark_bufferzone_test}(E) and (F) further illustrate that adjusting the threshold $\delta$ enables a controlled trade-off between accuracy and computational efficiency for the buffer-zone case. Compared with Figures~\ref{fig:benchmark_lattice_test}(E) and (F), the curves here are smoother, which is attributed to the gradual variation of the optical properties of the material.

\section{Conclusion}
\label{sec:conclusion}
Unlike the steady-state RTE, which characterizes the equilibrium distribution of radiation, the time-dependent RTE incorporates a temporal dimension to model the evolution of radiation propagation. Under standard diffusive scaling, the use of explicit time discretization introduces a restrictive CFL condition on the time step size, requiring $\Delta t\leq C\epsilon\Delta x$. To overcome this constraint, we employ an implicit time discretization, reformulating the time-dependent RTE as a sequence of steady-state RTEs with consistent cross sections. Crucially, the discrete solution space of the original time-dependent RTE aligns with that of these steady-state subproblems, consistent with the natural convergence of the time-dependent solution to the steady-state solution as $t \to \infty$.

For this family of steady-state RTEs, we first exploit the local low-rank structure in the angular domain to adaptively compress the solution space. This is accomplished by our newly developed ATFPS, originally devised for steady-state problems but readily extendable to the time-dependent case as argued above. Moreover, because we need to solve multiple steady-state RTEs sharing the same cross sections, we can explicitly assemble the solution operator that maps the source and boundary conditions to the corresponding solution, thereby achieving significant computational acceleration. This is realized via the RS method, which leverages the locality of the basis functions to perform a multilevel decomposition of the solution space and yields an explicit multilevel expression of the solution operator.

Furthermore, our framework enables efficient handling of changes in time step size, boundary conditions, and external source terms. These attributes collectively render our approach a powerful solver for multiscale time-dependent RTE problems, and even for inverse problems — a direction that will be the primary focus of our upcoming work.

\appendix
\section{TFPS basis functions}
\label{appendix:basis}

For any cell \( C \in \mathcal{C} \), the explicit expressions of the TFPS basis localized in cell \( C \) are shown in \eqref{eq:sec2:6x} and \eqref{eq:sec2:6y}. 
\begin{eqnarray}\label{eq:sec2:6x}
    \begin{aligned}
        \phi_{C}^{k}(x,y)=\xi_{C}^{(k)}\exp\{\lambda_{C}^{(k)}\sigma_{T,C}(x-x_{C}^{(k)})\},\quad 1\leq k\leq 4M,\\
        x_{C}^{(k)}=x_C^l,\quad 1\leq k\leq 2M;\quad x_{C}^{(k)}=x_C^r,\quad 2M+1\leq k\leq 4M;\\
    \end{aligned}
\end{eqnarray}  
\begin{eqnarray}\label{eq:sec2:6y}
    \begin{aligned}        
        \phi_{C}^{k}(x,y)=\xi_{C}^{(k)}\exp\{\lambda_{C}^{(k)}\sigma_{T,C}(y-y_C^{(k)})\},\quad 4M+1\leq k\leq 8M,\\
        y^{(k)}=y_{C}^{b},\quad 4M+1\leq k\leq 6M;\quad y^{(k)}=y_{C}^{t},\quad 6M+1\leq k\leq 8M;
    \end{aligned}
\end{eqnarray}
Within these expressions, \((\lambda_{C}^{(k)}, \xi_{C}^{(k)})\) for \( 1 \leq k \leq 4M \) and $(\lambda_{C}^{(k)},\xi_{C}^{(k)})$ for $4M+1\leq k\leq 8M$ are the eigenpairs corresponding to the following two matrices respectively,
\begin{equation}\label{eq:twomatrices}
    M_{C}^{x}=D^{-1}[\frac{\sigma_{s,C}}{\sigma_{T,C}}KW-I],\quad M_{C}^{y}=S^{-1}[\frac{\sigma_{s,C}}{\sigma_{T,C}}KW-I],
\end{equation}
where $D$, $S$, $W$ and $K$ are defined as:
\[D=\diag\{c_{1},c_{2},\dots,c_{4M}\},\  S=\diag\{s_{1},s_{2},\dots,s_{4M}\},\  W=\diag\{\omega_{1},\omega_{2},\dots,\omega_{4M}\},\]
$$K=
\begin{pmatrix}
    \kappa_{1,1} & \kappa_{1,2} &\dots & \kappa_{1,4M}\\
    \kappa_{2,1} & \kappa_{2,2} &\dots & \kappa_{2,4M}\\
    \vdots\\
    \kappa_{4M,1} & \kappa_{4M,2} &\dots & \kappa_{4M,4M}\\
\end{pmatrix}
.$$
We note that \(\xi_{C}^{(k)}\) are normalized such that \(\Vert \xi_{C}^{(k)} \Vert_{\infty} = 1\).

\section{The definition of notation $\mathcal{P}_{\delta,\mathfrak{i}}^{k}$}
\label{appendix:projection_operator}
For any $\mathfrak{i}\in\mathcal{I}_{\mathrm{in}}$, we define 
\[
U_{\mathfrak{i}}=\mathrm{span}\{\xi_{C}^{(k)}\}_{\mathfrak{i}\in C,k\in\mathcal{V}_{C,\mathfrak{i}}},\ U_{\delta,\mathfrak{i}}=\mathrm{span}\{\xi_{C}^{(k)}\}_{\mathfrak{i}\in C, k\in\mathcal{V}_{\delta,C,\mathfrak{i}}},\ \bar{U}_{\delta,\mathfrak{i}}=\mathrm{span}\{\xi_{C}^{(k)}\}_{\mathfrak{i}\in C, k\in\bar{\mathcal{V}}_{\delta,C,\mathfrak{i}}=\mathcal{V}_{C,\mathfrak{i}}\setminus\mathcal{V}_{\delta,C,\mathfrak{i}}}.
\]
Clearly $U_{\mathfrak{i}} = U_{\delta,\mathfrak{i}} + \bar{U}_{\delta,\mathfrak{i}}$, and by the linear independence result from \cite{AdaptiveTFPS}, we have $U_{\mathfrak{i}}=\mathbb{R}^{4M}$. 

Performing QR decomposition on $\{\xi_{C}^{(k)}\}_{\mathfrak{i}\in C, k\in\mathcal{V}_{\delta,C,\mathfrak{i}}}$ yields an orthonormal set $\{\chi_{\mathfrak{i}}^{(k)}\}_{k\in \mathcal{V}_{\delta}^{\mathfrak{i}}}$. Besides, we rewrite $\{\xi_{C}^{(k)}\}_{\mathfrak{i}\in C, k\in\bar{\mathcal{V}}_{\delta,C}^{\mathfrak{i}}}$ as $\{\chi_{\mathfrak{i}}^{(k)}\}_{k\in \bar{\mathcal{V}}_{\delta}^{\mathfrak{i}}}$. These vectors are normalized such that $\Vert\chi_{\mathfrak{i}}^{(k)}\Vert_{\infty}=1$. Thus $\{\chi_{\mathfrak{i}}^{(k)}\}_{k\in \mathcal{V}_{\delta,\mathfrak{i}}}$, $\{\chi_{\mathfrak{i}}^{(k)}\}_{k\in \bar{\mathcal{V}}_{\delta,\mathfrak{i}}}$ and $\{\chi_{\mathfrak{i}}^{(k)}\}_{k\in \mathcal{V}_{\mathfrak{i}}}$ serve as basis vectors for $U_{\delta,\mathfrak{i}}$, $\bar{U}_{\delta,\mathfrak{i}}$ and $U_{\mathfrak{i}}$, respectively. 
Hence we can write  
\[
U_{\delta,\mathfrak{i}}=\mathrm{span}\{\chi_{\mathfrak{i}}^{(k)}\}_{k\in \mathcal{V}_{\delta}^{\mathfrak{i}}},\quad \bar{U}_{\delta,\mathfrak{i}}=\mathrm{span}\{\chi_{\mathfrak{i}}^{(k)}\}_{k\in \bar{\mathcal{V}}_{\delta}^{\mathfrak{i}}},\quad U_{\mathfrak{i}}=\mathrm{span}\{\chi_{\mathfrak{i}}^{(k)}\}_{k\in \mathcal{V}^{\mathfrak{i}}}.
\]

For any $l \in \mathbb{R}^{4M}$, we expand it in these basis as  
\begin{equation}\label{eq:sec4:6}    
l=\sum_{k\in\mathcal{V}^{\mathfrak{i}}}\langle l,\chi_{\mathfrak{i}}^{(k)}\rangle_{\mathfrak{i}}\chi_{\mathfrak{i}}^{(k)},
\end{equation}
where $\langle l,\chi_{\mathfrak{i}}^{(k)}\rangle_{\mathfrak{i}}$ denotes the coefficient in the linear representation. Subsequently, the \textit{projection operator} $\mathcal{P}_{\delta,\mathfrak{i}}$ from $\mathbb{R}^{4M}$ (or equivalently, $U_{\mathfrak{i}}$) to its subspace $U_{\delta,\mathfrak{i}}$ can be defined as:
\begin{equation}\label{eq:sec4:14}
    \forall l\in\mathbb{R}^{4M},\quad \mathcal{P}_{\delta,\mathfrak{i}}l=\sum_{k\in \mathcal{V}_{\delta}^{\mathfrak{i}}}\langle l,\chi_{\mathfrak{i}}^{(k)}\rangle_{\mathfrak{i}}\chi_{\mathfrak{i}}^{(k)}.
\end{equation}
Its component is given by $\mathcal{P}_{\delta,\mathfrak{i}}^{k}l=\langle l,\chi_{\mathfrak{i}}^{(k)}\rangle_{\mathfrak{i}}$.

When $\mathfrak{i} \in \mathcal{I}_{\mathrm{b}}$, the \textbf{projection operator} $\mathcal{P}_{\delta,\mathfrak{i}}$ from $\mathbb{R}^{2M}$ onto the subspace $U_{\delta,\mathfrak{i}}$ and its component $\operatorname{Proj}_{\delta,\mathfrak{i}}^{k}$ can be defined in a similar manner.

\section{Proof of Lemma \ref{lemma:FG}}
\label{sec:proof_nested_solution_space}
For any $1\leq l\leq L$ and any function $f\in\mathcal{F}_{\delta}^{(l-1)}$, we define $f_{1}\in\mathcal{F}_{\delta}^{(l-1)}$ as the unique function satisfying the following conditions:
\begin{equation}\label{eq:f1_conditions}
    \begin{aligned}
        &\mathcal{P}_{\delta,\mathfrak{i}}^{k}[f_{1}(\mathbf{x}_{\mathfrak{i},\mathrm{mid}})]=\mathcal{P}_{\delta,\mathfrak{i}}^{k}[f(\mathbf{x}_{\mathfrak{i},\mathrm{mid}})],\quad\forall \mathfrak{i}\in\mathcal{I}_{\mathrm{in}}^{(l-1)}\setminus\mathcal{I}_{\mathrm{in}}^{(l)},\ \forall k\in\mathcal{V}_{\delta,\mathfrak{i}};\\
        &\mathcal{P}_{\delta,\mathfrak{i}}^{k}(f_{1}(\mathbf{x}_{\mathfrak{i},\mathrm{mid}})|_{C})=0,\quad\forall C\in\mathcal{C}^{(l)},\ \forall\mathfrak{i}\in\mathcal{I}^{(l)}\cap C,\ \forall k\in\mathcal{V}_{\delta,C,\mathfrak{i}}.\\
    \end{aligned}
\end{equation}
By the definition of $\mathcal{F}_{\delta}^{(l)}$ (cf. \eqref{eq:sec4:10}), it is clear that $f_{1}\in\mathcal{F}_{\delta}^{(l)}$.
Subsequently, we define $f_{2}$ as the remainder: $f_{2}\triangleq f-f_{1}$. Then, $f_{2}$ must satisfy:
\begin{equation}\label{eq:f2_conditions}
    \begin{aligned}
        &\mathcal{P}_{\delta,\mathfrak{i}}^{k}[f_{2}(\mathbf{x}_{\mathfrak{i},\mathrm{mid}})]=\mathcal{P}_{\delta,\mathfrak{i}}^{k}[f(\mathbf{x}_{\mathfrak{i},\mathrm{mid}})]-\mathcal{P}_{\delta,\mathfrak{i}}^{k}[f_{1}(\mathbf{x}_{\mathfrak{i},\mathrm{mid}})]=0,
        \quad\forall \mathfrak{i}\in\mathcal{I}_{\mathrm{in}}^{(l-1)}\setminus\mathcal{I}_{\mathrm{in}}^{(l)},\ \forall k\in\mathcal{V}_{\delta,\mathfrak{i}};\\
        &\mathcal{P}_{\delta,\mathfrak{i}}^{k}(f_{2}(\mathbf{x}_{\mathfrak{i},\mathrm{mid}})|_{C})=\mathcal{P}_{\delta,\mathfrak{i}}^{k}(f(\mathbf{x}_{\mathfrak{i},\mathrm{mid}})|_{C}),
        \quad\forall C\in\mathcal{C}^{(l)},\ \forall\mathfrak{i}\in\mathcal{I}^{(l)}\cap C,\ \forall k\in\mathcal{V}_{\delta,C,\mathfrak{i}}.\\
    \end{aligned}
\end{equation}
The conditions in \eqref{eq:f2_conditions} ensure that $f_{2}\in\mathcal{G}_{\delta}^{(l)}$. Therefore, we have successfully decomposed $f$ into $f=f_{1}+f_{2}$ with $f_{1}\in\mathcal{F}_{\delta}^{(l)}$ and $f_{2}\in\mathcal{G}_{\delta}^{(l)}$, which implies the following decomposition of the function spaces:
\begin{equation}\label{eq:sec2:25_refined}
    \mathcal{F}_{\delta}^{(l-1)}=\mathcal{F}_{\delta}^{(l)}+\mathcal{G}_{\delta}^{(l)}.
\end{equation}

To complete the proof of Lemma \ref{lemma:FG} based on the decomposition \eqref{eq:sec2:25_refined}, we also need to establish the dimension additivity:
\begin{equation}\label{eq:sec3:26_refined}
    \vert\mathcal{F}_{\delta}^{(l-1)}\vert=\vert\mathcal{F}_{\delta}^{(l)}\vert+\vert\mathcal{G}_{\delta}^{(l)}\vert.
\end{equation}
First, based on the recursive definition of $\mathcal{F}_{\delta}^{(l)}$ in \eqref{eq:sec4:10}, an equivalent definition is given by:
\[
\mathcal{F}_{\delta}^{(l)}=\{f\in\mathcal{F}_{\delta}\ |\ \mathcal{P}_{\delta,\mathfrak{i}}^{k}[f(\mathbf{x}_{\mathfrak{i},\mathrm{mid}})]=0,\ \forall\mathfrak{i}\in\mathcal{I}_{\mathrm{in}}\setminus\mathcal{I}_{\mathrm{in}}^{(l)},\ \forall k\in\mathcal{V}_{\delta,\mathfrak{i}}\}.
\]
Therefore, the dimension of $\mathcal{F}_{\delta}^{(l)}$ is calculated as:
\begin{equation}\label{eq:sec3:4}
    \vert\mathcal{F}_{\delta}^{(l)}\vert = \vert\mathcal{F}_{\delta}\vert - \sum_{\mathfrak{i}\in\mathcal{I}_{\mathrm{in}}\setminus\mathcal{I}_{\mathrm{in}}^{(l)}}\vert\mathcal{V}_{\delta,\mathfrak{i}}\vert = \sum_{\mathfrak{i}\in\mathcal{I}}\vert\mathcal{V}_{\delta,\mathfrak{i}}\vert - \sum_{\mathfrak{i}\in\mathcal{I}\setminus\mathcal{I}^{(l)}}\vert\mathcal{V}_{\delta,\mathfrak{i}}\vert = \sum_{\mathfrak{i}\in\mathcal{I}^{(l)}}\vert\mathcal{V}_{\delta,\mathfrak{i}}\vert.
\end{equation}
Next, according to the definition of $\mathcal{G}_{\delta}^{(l)}$ in \eqref{eq:sec4:11}, we compute its dimension:
\begin{equation}\label{eq:sec3:5}
\begin{aligned}
    \vert\mathcal{G}_{\delta}^{(l)}\vert &= \vert\mathcal{F}_{\delta}^{(l-1)}\vert - \sum_{C\in\mathcal{C}^{(l)}}\sum_{\mathfrak{i}\in\mathcal{I}^{(l)}\cap C}\vert\mathcal{V}_{\delta,C,\mathfrak{i}}\vert 
= \vert\mathcal{F}_{\delta}^{(l-1)}\vert - \sum_{\mathfrak{i}\in\mathcal{I}^{(l)}}\sum_{\mathfrak{i}\in C\in\mathcal{C}^{(l)}}\vert\mathcal{V}_{\delta,C,\mathfrak{i}}\vert \\
&= \vert\mathcal{F}_{\delta}^{(l-1)}\vert - \sum_{\mathfrak{i}\in\mathcal{I}^{(l)}}\vert\mathcal{V}_{\delta,\mathfrak{i}}\vert = \vert\mathcal{F}_{\delta}^{(l-1)}\vert - \vert\mathcal{F}_{\delta}^{(l)}\vert.\\
\end{aligned}
\end{equation}

The last step uses the dimension of $\mathcal{F}_{\delta}^{(l)}$ calculated above. Rearranging the terms, we immediately obtain equation \eqref{eq:sec3:26_refined}. This completes the proof of Lemma \ref{lemma:FG}.

\section{Proof of Lemma \ref{lemma:localization}}
\label{sec:proof_localization}
First, consider $l = 0$. By the definition of $\mathcal{F}_{\delta}$ in \eqref{eq:sec2:30}, the basis function $\phi_{C_{0},\mathfrak{i}_{0}}^{(0),k_{0}}$ admits the expansion
\begin{equation}\label{eq:RSF_basis_expansion}
    \phi_{C_{0},\mathfrak{i}_{0}}^{(0),k_{0}} = \sum_{C\in\mathcal{C}}\sum_{k\in\mathcal{V}_{\delta, C}}\alpha_{C}^{(k)}\phi_{C}^{k}
\end{equation}
where the coefficients $\alpha_{C}^{(k)}$ are to be determined. Since each $\phi_{C}^{k}$ is supported entirely within its associated cell $C$ for all $C \in \mathcal{C}$ and $k \in \mathcal{V}_{\delta, C}$, it follows immediately that $\alpha_{C}^{(k)} = 0$ whenever $C \neq C_{0}$. Consequently, the expansion simplifies to
\begin{equation}
    \phi_{C_{0},\mathfrak{i}_{0}}^{(0),k_{0}} = \sum_{k\in\mathcal{V}_{\delta, C_{0}}}\alpha_{C}^{(k)}\phi_{C_{0}}^{k}, 
\end{equation}
which confirms that $\phi_{C_{0},\mathfrak{i}_{0}}^{(0),k_{0}}$ is localized within the cell $C_{0}$.

Now assume $1 \leq l \leq L$. The definition of $\mathcal{F}_{\delta}^{(l)}$ can be reformulated as follows: for any $C_{0}\in\mathcal{C}^{(l)}$, $\mathfrak{i}_{0}\in\mathcal{I}^{(l)}\cap C_{0}$, and $k_{0}\in\mathcal{V}_{\delta,C_{0},\mathfrak{i}_{0}}$, the function $\phi_{C_{0},\mathfrak{i}_{0}}^{(l),k_{0}}$ is the unique element of $\mathcal{F}_{\delta}^{(l-1)}$ satisfying
\begin{equation}\label{eq:sec3:1}
    \begin{aligned}
        &\phi_{C_{0},\mathfrak{i}_{0}}^{(l),k_{0}} \in\mathcal{F}_{\delta}^{(l-1)}\\
        s.t.\quad & \mathcal{P}_{\delta,\mathfrak{i}_{0}}^{k_{0}}(\phi_{C_{0},\mathfrak{i}_{0}}^{(l),k_{0}}(\bar{\mathbf{x}}_{\mathfrak{i}_{0}})|_{C_{0}})=1;\\
        &\mathcal{P}_{\delta,\mathfrak{i}}^{k}(\phi_{C_{0},\mathfrak{i}_{0}}^{(l),k_{0}}(\bar{\mathbf{x}}_{\mathfrak{i}})|_{C})=0,\quad \forall\ C\in\mathcal{C}^{(l)},\ \mathfrak{i}\in \mathcal{I}^{(l)}\cap C, \  k\in\mathcal{V}_{\delta,C,\mathfrak{i}},\ (C,\mathfrak{i},k)\ne (C_{0},\mathfrak{i}_{0},k_{0});\\
        & \mathcal{P}_{\delta,\mathfrak{i}}^{k}[\phi_{C_{0},\mathfrak{i}_{0}}^{(l),k_{0}}(\bar{\mathbf{x}}_{\mathfrak{i}})]=0,\quad \forall \mathfrak{i}\in \mathcal{I}^{(l-1)}\setminus\mathcal{I}^{(l)},\quad \forall k\in\mathcal{V}_{\delta,\mathfrak{i}}.
    \end{aligned}
\end{equation}
Since $\phi_{C_{0},\mathfrak{i}_{0}}^{(l),k_{0}} \in \mathcal{F}_{\delta}^{(l-1)}$, it can be expressed as
\begin{equation}
    \phi_{C_{0},\mathfrak{i}_{0}}^{(l),k_{0}} = \sum_{C\in\mathcal{C}^{(l-1)}}\sum_{\mathfrak{i}\in\mathcal{I}^{(l-1)}\cap C}\sum_{k\in\mathcal{V}_{\delta,C,\mathfrak{i}}}\alpha_{C,\mathfrak{i}}^{(l-1),k}\phi_{C,\mathfrak{i}}^{(l-1),k}
\end{equation}
By the induction hypothesis, each $\phi_{C,\mathfrak{i}}^{(l-1),k}$ is localized within its cell $C$. Therefore, the constraints in \eqref{eq:sec3:1} imply that $\alpha_{C,\mathfrak{i}}^{(l-1),k} = 0$ unless $C \subset C_{0}$. Hence, the expansion reduces to
\begin{equation}\label{eq:sec3:14}
    \phi_{C_{0},\mathfrak{i}_{0}}^{(l),k_{0}} = \sum_{C\in C_{0}\cap\mathcal{C}^{(l-1)}}\sum_{\mathfrak{i}\in\mathcal{I}^{(l-1)}\cap C}\sum_{k\in\mathcal{V}_{\delta,C,\mathfrak{i}}}\alpha_{C,\mathfrak{i}}^{(l-1),k}\phi_{C,\mathfrak{i}}^{(l-1),k}, 
\end{equation}
which shows that $\phi_{C_{0},\mathfrak{i}_{0}}^{(l),k_{0}}$ is supported only within $C_{0}$.

By mathematical induction, it follows that $\phi_{C_{0},\mathfrak{i}_{0}}^{(l),k_{0}}$ is localized within the cell $C_{0}$ for all $0 \leq l \leq L$. An analogous argument applies to the basis functions of $\mathcal{G}_{\delta}^{(l)}$, establishing their localization as well.

\section{Proof of Lemma \ref{lemma:twoleveldecomposition}}
\label{append:twoleveldecomposition}

Let $\mathbf{v}\in\mathbb{R}^{\vert\mathcal{F}_{\delta}^{(l)}\vert}$ be an arbitrary vector. There exists a function $f\in\mathcal{F}_{\delta}^{(l)}$ such that $\mathbf{v}=\big(\mathbf{f}_{\mathcal{C}^{(l)},\mathrm{b}}^{T},[\mathbf{f}]_{\mathcal{I}_{\mathrm{in}}^{(l)}}\big)^{T}$. Consequently, we have
\begin{equation}\label{eq:sec3:9}
  \mathbf{f}_{\mathcal{C}^{(l)}} = B_{\delta}^{(l),-1}\mathbf{v}.
\end{equation}
Furthermore, by Lemma \ref{lemma:FG}, $f$ admits a decomposition $f = f_{1} + f_{2}$, where $f_{1}\in\mathcal{F}_{\delta}^{(l+1)}$ and $f_{2}\in\mathcal{G}_{\delta}^{(l+1)}$. This implies
\begin{equation}\label{eq:sec3:10}
  \mathbf{f}_{\mathcal{C}^{(l)}} = \mathbf{f}_{1,\mathcal{C}^{(l)}} + \mathbf{f}_{2,\mathcal{C}^{(l)}}.
\end{equation}
We now proceed to determine the expressions for $\mathbf{f}_{1,\mathcal{C}^{(l)}}$ and $\mathbf{f}_{2,\mathcal{C}^{(l)}}$.From equation \eqref{eq:f2_conditions}, we have
\[
[\mathbf{f}]_{\mathcal{I}^{(l)}\setminus\mathcal{I}^{(l+1)}} = [\mathbf{f}_{2}]_{\mathcal{I}^{(l)}\setminus\mathcal{I}^{(l+1)}}.
\]
It follows that
\begin{equation}\label{eq:sec3:11}
\mathbf{f}_{2,\mathcal{C}^{(l)}} = Q_{\delta}^{(l)}[\mathbf{f}_{2}]_{\mathcal{I}^{(l)}\setminus\mathcal{I}^{(l+1)}} = Q_{\delta}^{(l)}\check{R}_{\delta}^{(l)}\big(\mathbf{f}_{\mathcal{C}^{(l)},\mathrm{b}}^{T},[\mathbf{f}]_{\mathcal{I}_{\mathrm{in}}^{(l)}}\big)^{T} = Q_{\delta}^{(l)}\check{R}_{\delta}^{(l)}\mathbf{v}.
\end{equation}
Next, we observe that
\begin{equation*}
\begin{aligned}
  \big(\mathbf{f}_{1,\mathcal{C}^{(l+1)},\mathrm{b}}^{T},[\mathbf{f}_{1}]_{\mathcal{I}_{\mathrm{in}}^{(l+1)}}\big)^{T} & = \big(\mathbf{f}_{\mathcal{C}^{(l+1)},\mathrm{b}}^{T},[\mathbf{f}]_{\mathcal{I}_{\mathrm{in}}^{(l+1)}}\big)^{T} - \big(\mathbf{f}_{2,\mathcal{C}^{(l+1)},\mathrm{b}}^{T},[\mathbf{f}_{2}]_{\mathcal{I}_{\mathrm{in}}^{(l+1)}}\big)^{T} = R_{\delta}^{(l)}\big(\mathbf{f}_{\mathcal{C}^{(l)},\mathrm{b}}^{T},[\mathbf{f}]_{\mathcal{I}_{\mathrm{in}}^{(l)}}\big)^{T} - B_{\delta}^{(l+1)}\mathbf{f}_{2,\mathcal{C}^{(l)}} \\
  &= R_{\delta}^{(l)}\mathbf{v} - B_{\delta}^{(l+1)}Q_{\delta}^{(l)}\check{R}_{\delta}^{(l)}\mathbf{v}.
\end{aligned}
\end{equation*}
Thus, we obtain
\begin{equation}\label{eq:sec3:12}
  \mathbf{f}_{1,\mathcal{C}^{(l)}} = P_{\delta}^{(l)}\mathbf{f}_{1,\mathcal{C}^{(l+1)}} = P_{\delta}^{(l)}\mathcal{B}_{\delta}^{(l+1),-1}\big(\mathbf{f}_{1,\mathcal{C}^{(l+1)},\mathrm{b}}^{T},[\mathbf{f}_{1}]_{\mathcal{I}_{\mathrm{in}}^{(l+1)}}\big)^{T} = P_{\delta}^{(l)}\mathcal{B}_{\delta}^{(l+1),-1}(R_{\delta}^{(l)} - B_{\delta}^{(l+1)}Q_{\delta}^{(l)}\check{R}_{\delta}^{(l)})\mathbf{v}.
\end{equation}
Substituting equations \eqref{eq:sec3:9}, \eqref{eq:sec3:11}, and \eqref{eq:sec3:12} into \eqref{eq:sec3:10} yields
\[
B_{\delta}^{(l),-1}\mathbf{v} = P_{\delta}^{(l)}\mathcal{B}_{\delta}^{(l+1),-1}(R_{\delta}^{(l)} - B_{\delta}^{(l+1)}Q_{\delta}^{(l)}\check{R}_{\delta}^{(l)})\mathbf{v} + Q_{\delta}^{(l)}\check{R}_{\delta}^{(l)}\mathbf{v}.
\]
Since $\mathbf{v}$ is arbitrary, the proof is complete.

\section*{Acknowledgments}
M. Tang  is supported by the Strategic Priority Research Program of Chinese Academy of Sciences Grant No.XDA25010401; NSFC12031013, Shanghai pilot innovation project 21JC1403500 and Mevion Medical Systems, Inc., Kunshan. L. Zhang is partially supported by the NSFC grant 12271360, and the Fundamental Research Funds for the Central Universities.



\end{document}